\documentclass[reqno]{amsart}
\usepackage{enumitem}

\usepackage[all]{xy}

\usepackage{graphicx}
\usepackage{tikz}
\usetikzlibrary{intersections}
\usetikzlibrary {arrows.meta}
\usetikzlibrary{calc}
\usetikzlibrary{spath3}
\usetikzlibrary{decorations.markings}
\usetikzlibrary{decorations.pathreplacing}
\tikzset{
    position label/.style={
       text height = 1.5ex,
       text depth = 1ex
    },
   brace/.style={
     decoration={brace, mirror},
     decorate
   }
}
\usepackage{mathrsfs}
\usepackage{mathtools}
\usepackage{amsmath}
\usepackage{amssymb}
\usepackage{caption}
\usepackage{subcaption}
\usepackage{hyperref}

\hypersetup{%
  colorlinks=true,
  linkcolor=blue
}

\mathtoolsset{showonlyrefs=true}
\numberwithin{equation}{section}

\newtheorem{thm}{Theorem}[section]
\newtheorem{lem}[thm]{Lemma}
\newtheorem{pro}[thm]{Proposition}

\newtheorem{cor}[thm]{Corollary}
\newtheorem{claim}{Claim}

\theoremstyle{definition}
\newtheorem{dfn}[thm]{Definition}
\newtheorem{rk}[thm]{Remark}

\def\cal#1{\fam2#1}

\def\R{{\mathbb R}}
\def\Z{{\mathbb Z}}

\def\B{{\mathbb B}}

\def\bb{\begin}
\def\be{\begin{equation}}
\def\ee{\end{equation}}
\def\bea{\begin{eqnarray}}
\def\eea{\end{eqnarray}}
\def\beaa{\begin{eqnarray*}}
\def\eeaa{\end{eqnarray*}}

\def\pr{{\bf pr}}

\def\bb{\begin}
           \def\ea{\end{array}}
          \def\ec{\end{center}}
     \def\ed{\end{description}}
\def\be{\bb{equation}}        \def\ee{\end{equation}}
\def\bea{\bb{eqnarray}}       \def\eea{\end{eqnarray}}
\def\beaa{\bb{eqnarray*}}     \def\eeaa{\end{eqnarray*}}
 \def\et{\end{thebibliography}}

        \def\d{\delta}
      \def\e{\varepsilon}

\def\D{{\cal D}}           
           
   \def\B{{\cal B}}

\def\bar2{\doublebar}

\def\qed {\hfill $\Box$\vskip5pt}

\def\dist{{\rm dist}}

\def\Orb{{\rm Orb}}

\newcommand{\capitalize}[1]{\expandafter\capitalizetwo#1\relax}
\def\capitalizetwo#1#2\relax{\MakeUppercase{#1}\MakeLowercase{#2}}
\title{\capitalize{Multidimensional} $C^0$ \MakeLowercase{transversality and the shadowing property for Axiom A diffeomorphisms}}

\author{Sogo Murakami\\
Graduate School of Mathematical Sciences, University of Tokyo, Japan}

\address{Graduate School of Mathematical Sciences, University of Tokyo, Japan}
\email{murakami-sogo880@g.ecc.u-tokyo.ac.jp}



\begin{document}

\begin{abstract}
	Petrov and Pilyugin (2015) generalized a notion of $C^0$ transversality of Sakai (1995) using smooth curves.
	Their definition involves only continuous maps from $\R^n$ to a manifold, which is a purely topological one.
	They also provided a sufficient condition for the $C^0$ transversality in terms of homological nature.
	In this paper, we prove that such a homological condition of Axiom A diffeomorphisms is sufficient for enjoying the shadowing property.
	Moreover, it is proved that the $C^0$ transversality of Axiom A diffeomorphisms with codimension one basic sets implies the homological condition.
\end{abstract}

\maketitle{}

\section{Introduction}
The theory of shadowing property has been developed in the study of hyperbolic dynamical systems (see \cite{S.P.} or \cite{R} for instance).
Many results concerning the relationship between various shadowing and the stability has been obtained for both diffeomorphisms and flows.
One of the most important studies among them is due to Pilyugin and Tikhomirov \cite{P.S.T.}, who proved that the Lipschitz shadowing property and the $C^1$ structural stability are equivalent for $C^1$ diffeomorphisms on a compact manifold.

Related to this research, it should be noticed that sufficient conditions for the shadowing property has not been completely understood.
In this direction, Sakai \cite{ShadowingAnTransversality} introduced the $C^0$ transversality condition for two smooth curves in closed surface and gave a rough proof for Axiom A diffeomorphisms on a closed surface, claiming the condition is necessary and sufficient for the shadowing property.
The complete proof of it was given by Sakai and Pilyugin \cite{P.S.trans} in $2007$.
In $2015$, Petrov and Pilyugin generalized the notion of $C^0$ transversality to arbitrary $C^0$ maps from $\R^n$ to a manifold \cite[Definition $1$, $2$, $3$]{PETROV2015}, giving some properties of $C^0$ transversality.
Indeed, they introduced a condition using homological condition (that is called {\it the condition $T(V_A, V_B)$}) and proved that the condition $T(V_A, V_B)$ is sufficient for $C^0$ transversality.

In this paper, we define a condition quite close to the condition $T(V_A, V_B)$ called {\it the $T^{s, u}$-condition}, and prove that, for an Axiom A diffeomorphism $f$ with the nonwandering set consisting of attractors, repellers and codimension one basic sets, three conditions, the $T^{s, u}$-condition, the condition $T(V_A, V_B)$ and the $C^0$ transversality condition, are equivalent.
Moreover, it is proved that if an Axiom A diffeomorphism $f$ satisfies the $T^{s, u}$-condition, then $f$ has the shadowing property.
As an application, we prove that if an Axiom A diffeomorphism $f$ has the nonwandering set consisting of attractors, repellers and codimension one basic sets, then the $C^0$ transversality condition is sufficient for the shadowing property.
This extends Sakai's result mentioned above for the sufficient condition in dimension $3$.

Let $M$ be an $m$-dimensional $C^\infty$ closed manifold with Riemannian metric $\dist$.
Let $f$ be a $C^1$ diffeomorphism on $M$ satisfying Axiom A.
Denote by $\Omega(f)$ the nonwandering set of $f$.
It is well known \cite{Wen2016DifferentiableDS} that
\[
	\bigcup_{x \in \Omega(f)} W^s(x)
	= \bigcup_{x \in \Omega(f)} W^u(x)
	= M.
\]
Then, for all $x \in M$ and $\sigma = s, u$, there exists $p \in \Omega(f)$ satisfying $x \in W^\sigma(p)$.
Then $W^\sigma(x) = W^\sigma(p)$ ($\sigma = s, u$) is an immersed manifold and for each $\sigma = s, u$, there exists a standard immersion $h^\sigma : \R^{\dim W^\sigma(x)} \to M$ such that $h^\sigma (\R^{\dim W^\sigma(x)}) = W^\sigma(x)$.


The first theorem gives a sufficient condition for Axiom A diffeomorphisms on a closed manifold to have the shadowing property (see Section \ref{sec.preliminaries} for the definition).
For an Axiom A diffeomorphism $f$, roughly speaking, we say that $f$ satisfies {\it the $T^{s, u}$-condition} when, for all $p \in M$, the inclusion from a subset of $\partial W^u_{\rm loc}(p)$ homeomorphic to $(m - \dim W^s(p) - 1)$-dimensional sphere to $B(\e, p) \setminus W^s_{\rm loc}(p)$ with some $\e > 0$ induces a nontrivial homomorphism between homology groups
(the precise definition of the $T^{s, u}$-condition will be given in Section \ref{sec.preliminaries}, where note that $T^{s, u}$ and $T^{u, s}$ have different meanings).
\begin{thm}\label{thm.homologyshadowing}
	Let $f$ be a $C^1$ Axiom A diffeomorphism on $M$.
	If $f$ satisfies the $T^{s, u}$ or $T^{u, s}$-condition,
	then $f$ has the shadowing property.
\end{thm}
This Theorem is proved via the following proposition.
\begin{pro}\label{pro.starimpliesnocycle}
	Let $f$ be an Axiom A diffeomorphism on $M$. If $f$ satisfies the $T^{s, u}$-condition, then $f$ satisfies the no-cycles condition.
\end{pro}

Let us recall the definition of $C^0$ transversality given by Petrov and Pilyugin \cite{PETROV2015}.
For topological space $A$,
we consider the $C^0$ uniform metric on the space of continuous mappings from $A$ to $M$; i.e.,
\[
	\lvert h_1, h_2 \rvert_{C^0} = \sup \{\dist(h_1(x), h_2(x)) ; x \in A\}.
\]
In the following Definitions \ref{dfn.dessential}, \ref{dfn.c0transverseatthepair} and \ref{dfn.c0transverse},
let $A$ and $B$ be topological spaces and let $h_1 : A \to M$ and $h_2 : B \to M$ be continuous mappings.
\begin{dfn}\label{dfn.dessential}
	Let $C \subset A$, $D \subset B$ and $\d > 0$ be given.
	Then we say that the intersection $h_1(C) \cap h_2(D)$ is $\d$-essential if
	\[\widetilde{h}_1(C) \cap \widetilde{h}_2(D) \neq \emptyset\]
	for all $\widetilde{h}_i$, $i = 1, 2$, such that $\lvert h_1, \widetilde{h}_1 \rvert_{C^0} \leq \d$ and $\rvert h_2, \widetilde{h}_2 \rvert_{C^0} \leq \d$.
\end{dfn}
\begin{dfn}\label{dfn.c0transverseatthepair}
	Assume $h_1(a) = h_2(b)$ for some $a \in A$ and $b \in B$.
	We say that $h_1$ and $h_2$ are $C^0$ transverse at the pair $(a, b)$
	if for all neighborhoods $U(a) \subset A$ and $U(b) \subset B$ of $a$ and $b$, respectively,
	there exists $\d > 0$ such that the intersection $h_1(U(a)) \cap h_2(U(b))$ is $\d$-essential.
\end{dfn}
\begin{dfn}\label{dfn.c0transverse}
	We say that $h_1$ and $h_2$ are $C^0$ transverse
	if $h_1$ and $h_2$ are $C^0$ transverse at any pair $(a, b) \in A \times B$ with $h_1(a) = h_2(b)$.
\end{dfn}
We say that $f$ satisfies the {\it $C^0$ transversality condition}
if $W^s(p)$ and $W^u(q)$ are $C^0$ transverse for all $p, q \in \Omega(f)$.
Our second main theorem is a partial generalization of the main theorem of \cite{P.S.trans}.
\begin{thm}\label{thm.3dim}
	Let $f$ be a $C^1$ Axiom A diffeomorphism on $M$.
	Assume that
	\[
		\dim W^s(x), \, \dim W^u(x) \in \{ 0, 1, m-1, m \}
	\]
	for all $ x \in M$.
	If $f$ satisfies the $C^0$ transversality condition,
	then $f$ has the shadowing property.
\end{thm}
The hypothesis of this theorem can be dropped in dimension $\leq 3$.
\begin{cor}
	Let $M$ be a closed manifold with dimension less than or equal to $3$ and let $f$ be a $C^1$ Axiom A diffeomorphism on $M$.
	If $f$ satisfies the $C^0$ transversality condition,
	then $f$ has the shadowing property.
\end{cor}

\section{Preliminaries}\label{sec.preliminaries}
In this section, we first give the definition of the shadowing property for homeomorphisms on $M$.
For a homeomorphism $f : M \to M$ and $d > 0$,
we call a sequence
\[
	\xi = \{x_k ; k \in \Z\}
\]
a {\it $d$-pseudotrajectory} if
\[
	\dist(f(x_k), x_{k+1}) < d, \quad k \in \Z.
\]
$d$-pseudotrajectory $\xi$ is {\it $\e$-shadowed} by $x \in M$ when
\[
	\dist(x_k, f^k(x)) < \e, \quad k \in \Z.
\]
We say that $f$ has {\it the shadowing property} is for every $\e>0$, there is $d > 0$ such that every $d$-pseudotrajectory $\xi$ is $\e$-shadowed by some point.

In \cite{PETROV2015}, Petrov and Pilyugin gave a sufficient condition for an intersection of continuous images of two sets to be $\d$-essential.
Now,
we recall the condition and define a similar condition called the $T^{s, u}$-condition in line with the sufficient condition in \cite{PETROV2015}.
Let $A$ and $B$ be smooth manifolds with dimension $\iota$ and $\kappa$, respectively,
and let $h_1 : A \to M$ and $h_2 : B \to M$ be $C^1$ immersions.

Since $h_1$ is an immersion, for every $a \in A$, there exists a neighborhood $U(a) \subset A$ of $a$ such that $h_1 \vert_{U(a)}$ is an embedding.
Then there are a neighborhood $V$ of $h_1(a) \in M$ and diffeomorphism $\varphi : V \to (-1, 1)^m$ satisfying
\[
	\varphi(h_1(U(a)) \cap V) = (-1, 1)^\iota \times \{ 0 \}^{m-\iota}.
\]
Such a diffeomorphism is called a {\it local chart of $U(a)$ around $a$}.
\begin{dfn}\label{dfn.TUaUbcondition}
	Let $p \in M$ be such that $p = h_1(a) = h_2(b)$ for some $a \in A$ and $b \in B$.
	Let $U(a) \subset A$ and $U(b) \subset B$ be neighborhoods of $a$ and $b$, respectively.
	If $\varphi : U \to (-1, 1)^m$ is a local chart of $U(a)$ around $a$,
	then we say that $h_1(U(a))$ and $h_2(U(b))$ satisfy the strong $T(U(a), U(b))$-condition (resp. condition $T(U(a), U(b))$) at $p$ when one of the following properties hold:
	\begin{itemize}[leftmargin=0.6cm, itemindent=0cm]
		\item $\iota = m$.
		\item $\iota < m$ and there exists an embedded $(m-\iota)$-dimensional closed ball $B_u \subset U(b)$ (resp. an image of continuous map $B_u \subset U(b)$ of an $(m-\iota)$-dimensional closed ball) centered at $p$ such that
		      $\varphi \circ h_2(\partial B_u) \subset (-1, 1)^m \setminus ((-1, 1)^\iota \times \{0\}^{m-\iota})$
		      and the induced homomorphism between reduced homology groups
		      \[
			      (\varphi \circ h_2)_* : \tilde{H}_{m-\iota-1}(\partial B_u) \to \tilde{H}_{m-\iota-1}((-1, 1)^m \setminus ((-1, 1)^\iota \times \{0\}^{m-\iota}))
		      \]
		      is nontrivial.
	\end{itemize}
	
\end{dfn}
Note that this definition is independent of the choice of the local chart.
\begin{rk}\label{rk.PETROV2015prop2}
	In \cite[Proposition 2]{PETROV2015}, Petrov and Pilyugin proved that if the condition $T(U(a), U(b))$ holds, then the intersection is $\d$-essential for sufficiently small $\d> 0$.
	The strong $T(U(a), U(b))$-condition clearly implies the condition $T(U(a), U(b))$. We will prove that for $\iota, \kappa = 0, 1, m-1, m$, if an intersection $h_1(U(a)) \cap h_2(U(b))$ is $\d$-essential, then the intersection satisfies the strong $T(U(a), U(b))$-condition. Thus, these three conditions are equivalent when $\iota, \kappa = 0, 1, m-1, m$.
\end{rk}
For simplicity, we just use a terminology the $T(U(a), U(b))$-condition instead of the strong $T(U(a), U(b))$-condition.
\begin{rk}\label{rk.homology}$\,$
	\begin{itemize}[leftmargin=0.6cm, itemindent=0cm]
		\item[(1)]
		      It is well-known that if there exists $B_u \subset U(b)$ homeomorphic to $B^{m-\iota}$,
		      then $\kappa \geq m-\iota$ holds \cite[Theorem 2B.3.]{HatcherAlgTop}.
		      Therefore, if $h_1(U(a))$ and $h_2(U(b))$ satisfy the $T(U(a), U(b))$-condition at $p$,
		      then $\iota + \kappa \geq m$.
		\item[(2)]
		      If $h_1(U(a))$ and $h_2(U(b))$ have a transversal intersection at $h_1(a) = h_2(b)$,
		      then there is a local chart $\varphi$ of $U(a)$ around $a$ such that
		      \[
			      \varphi^{-1}(\{ 0 \}^\iota \times (-1, 1)^{m-\iota}) \subset h_2(U(b)).
		      \]
		      In this case, by taking an embedded $(m-\iota)$-dimensional closed ball
		      \[
			      B_u \subset h_2^{-1} \circ \varphi^{-1}(\{ 0 \}^\iota \times (-1, 1)^{m-\iota}) \subset U(b)
		      \]
		      centered at $b$,
		      we can prove that $h_1(U(a))$ and $h_2(U(b))$ satisfy the $T(U(a), U(b))$-condition at $h_1(a) = h_2(b)$.
	\end{itemize}
\end{rk}

\begin{dfn}
	Assume $h_1(a) = h_2(b)$ for some $a \in A$ and $b \in B$.
	Then a pair $(h_1, h_2)$ is called $T$-pair at $h_1(a) = h_2(b)$ if for all neighborhoods $U(a) \subset A$ and $U(b) \subset B$ of $a$ and $b$, respectively,
	$h_1(U(a)))$ and $h_2(U(b))$ satisfy the $T(U(a), U(b))$-condition at $h_1(a) = h_2(b)$.
	Moreover, if $(h_1, h_2)$ is $T$-pair at any point $p \in M$ such that $p = h_1(a) = h_2(b)$ for some pair $(a, b) \in A \times B$, then we say that $h_1$ and $h_2$ satisfy the $\tilde{T}$-condition.
\end{dfn}
We say that $W^s(p)$ and $W^u(q)$ (resp. $W^u(q)$ and $W^s(p)$) satisfy the {\it $T$-condition} if the standard immersions $h^s : \R^{\dim W^s(p)} \to W^s(p)$ and $h^u : \R^{\dim W^u(p)} \to W^u(q)$ (resp. $h^u$ and $h^s$) satisfy the $\tilde{T}$-condition.

We say that $f$ satisfies the {\it $T^{s, u}$-condition} (resp. {\it $T^{u, s}$-condition}) if $W^s(p)$ and $W^u(q)$ (resp. $W^u(q)$ and $W^s(p)$) satisfy the $T$-condition for all $p, q \in \Omega(f)$.
In \cite{PETROV2015}, it is proved that if $f$ satisfies the $T^{s, u}$-condition then $f$ has the $C^0$ transversality condition.

\section{Invariant disk families in neighborhoods of basic sets}\label{sec.prepforthm.homologyshadowing}
Let $\Lambda \subset M$ be a compact locally maximal hyperbolic set of $f$.
It is well known (\cite[p.131]{NbdsofHypsets}) that for some neighborhood $U$ of $\Lambda$, we may extend the hyperbolic splitting $T_{\Lambda}M = E^s \oplus E^u$ to vector bundles $T_{U}M = \widetilde{E}^s \oplus \widetilde{E}^u$ such that
\[
	Tf(\widetilde{E}^s_x) = \widetilde{E}^s_{f(x)}, \quad x \in U \cap f^{-1}(U),
\]
and
\[
	Tf(\widetilde{E}^u_x) = \widetilde{E}^u_{f(x)}, \quad x \in U \cap f^{-1}(U).
\]
Let $\| \cdot \|$ be the norm on $TM$ induced by the Riemannian metric. For $(x, v) \in T_{U}M$, define
\[
	\lvert v \rvert = \max \{ \| v_s \|, \| v_u \| \}, \quad v \in T_x M, x \in U,
\]
where $v_s \in \widetilde{E}^s_{x}$, $v_u \in \widetilde{E}^u_{x}$ and $v_s + v_u = v$.
Also, it is well known (\cite[Theorem 4.5]{NbdsofHypsets}) that (replacing $f$ by $f^N$ for sufficiently large $N$ if necessary),
there are $\lambda \in (0, 1)$ and families of embedded closed balls
\[
	x \in \widetilde{W}^\sigma(x) \subset M, \quad x \in U, \sigma = s, u,
\]
satisfying the following properties:
\begin{itemize}[leftmargin=0.6cm, itemindent=0cm]
	\item $\widetilde{W}^s(x)$ and $\widetilde{W}^u(x)$ depend continuously on $x$;
	\item If $x \in \Lambda$, then $\widetilde{W}^\sigma(x) = W^\sigma_{\rm loc}(x)$ for $\sigma = s, u$;
	\item $f(\widetilde{W}^s(x)) \subset \widetilde{W}^s(f(x))$ if $x \in U \cap f^{-1}(U)$;
	\item $f^{-1}(\widetilde{W}^u(x)) \subset \widetilde{W}^u(f^{-1}(x))$ if $x \in U \cap f(U)$.
\end{itemize}
Define metrics $d^\sigma_x$ on $\widetilde{W}^\sigma(x)$, $\sigma = s, u$, by
\[
	d^\sigma_x(y, z) = \lvert \exp_x^{-1}(y) - \exp_x^{-1}(z) \rvert, \quad y, z \in \widetilde{W}^\sigma(x), \quad x \in U_i.
\]
Then we also have the following properties:
\begin{itemize}[leftmargin=0.6cm, itemindent=0cm]
	\item
	      $d^s_{f(z)}(f(x), f(y)) \leq \lambda d^s_z(x, y)$ for all $x, y \in \widetilde{W}^s(z)$ if $z \in U \cap f^{-1}(U)$;
	\item $d^u_{f^{-1}(z)}(f^{-1}(x), f^{-1}(y)) \leq \lambda d^u_z(x, y)$ for all $x, y \in \widetilde{W}^u(z)$ if $z \in U \cap f(U)$.
\end{itemize}

Let $\dim \widetilde{W}^\sigma (x)$ be the dimension of the closed disk $\widetilde{W}^\sigma (x)$.
There exists $C_0 > 0$ such that
\begin{equation}
	C_0^{-1}\dist(x, y) < d^\sigma_z(x, y) < C_0\dist(x, y)\label{eq.constofequivmetric}
\end{equation}
for all $x, y \in \widetilde{W}^\sigma(z)$ and $z \in U_i$.
Then if $\dist(x, y)$ with $x, y \in U_i$ is small enough,
there exists a unique point
\[
	[x, y] = \widetilde{W}^u(x) \cap \widetilde{W}^s(y).
\]
For $x \in U$, $\sigma = s, u$,
and a small positive number $\alpha$,
we set
\[
	\widetilde{W}^\sigma(x, \alpha) = \{ y \in \widetilde{W}^\sigma(x) : d_x^\sigma(x, y) \leq \alpha \}, \quad \sigma = s, u.
\]
For any point $x \in U_i$,
and small positive numbers $\alpha$ and $\beta$, let
\[
	\B(x, \alpha, \beta) = \{ y \in U :  \widetilde{W}^s(y, \alpha) \cap \widetilde{W}^u(x, \beta) \neq \emptyset \}.
\]
Note that $\B(x, \alpha, \beta)$ is a neighborhood of $x \in M$.
\begin{claim}\label{claim.condonWi}
	There exist a closed neighborhood $W \subset U \cap f^{-1}(U)$ of $\Lambda$ and constants $\overline{\alpha}, C > 0$ satisfying the following properties:
	\begin{enumerate}[leftmargin=0.8cm, itemindent=0cm]
		\item[$(1)$] $\B(x, \alpha, \beta) \subset U \cap f^{-1}(U)$ for all $x \in W$ and $\alpha, \beta \in (0, \overline{\alpha})$.
		\item[$(2)$] For every $\e > 0$, there exists $\d > 0$ such that if $x, x' \in W$ and $y, y' \in U$ satisfy $\widetilde{W}^s(y, \overline{\alpha}) \cap \widetilde{W}^u(x, \overline{\alpha}) \neq \emptyset$ and $\dist(x, x'), \dist(y, y') < \d$, then
		      \[
			      d^u_x([x, y], [x, y']), \, d^s_y([x, y], [x', y]) < \e.
		      \]
		\item[$(3)$]
		      Let $x \in W$ satisfy $f^k(x) \in W$ for all $k \ge 0$. 
		      Then, $y \in W^s(q, C)$ for all $q \in \Lambda$ and $y \in U$ satisfying $W^s(q, \overline{\alpha}) \cap \widetilde W^s(y, \overline{\alpha}) \ni x$.
		      Here it is clear that there exists such a point $q$ by the choice of $x$.  
	\end{enumerate}
\end{claim}
\begin{proof}
	\begin{enumerate}[leftmargin=*, itemindent=0.6cm, itemsep=0.2cm]
		\item[(1)]
		      Take $W$ so small that $W \subset U \cap f^{-1}(U)$.
		      There exists $\e > 0$ such that the $\e$-neighborhood $B(W, \e)$ satisfies $B(W, \e) \subset U \cap f^{-1}(U)$.
		      Then for all $x \in W$ and $y \in \B(x, \alpha, \beta)$ with $\alpha, \beta \in (0, \overline{\alpha})$, by \eqref{eq.constofequivmetric},
		      \begin{align}
			      \dist(x, y)
			       & \leq \dist(x, [x, y]) + \dist([x, y], y)           \\
			       & \leq C_0 d^u_x(x, [x, y]) +  C_0 d^s_y([x, y], y)  \\
			       & \leq C_0 \overline{\alpha} + C_0 \overline{\alpha}
			      = 2C_0 \overline{\alpha},\label{eq.distxy}
		      \end{align}
		      where $C_0$ is the constant given by \eqref{eq.constofequivmetric}.
		      Take $\overline{\alpha}$ so small that $\overline{\alpha} < \e/(2C_0)$.
		      Then we obtain $y \in B(x, \e) \subset U \cap f^{-1}(U)$, from which (1) follows.
		\item[(2)]
		      For sufficiently small $r > 0$, if $z, w \in U$ satisfy $\dist(z, w) \leq r$, then $\widetilde{W}^u(z) \cap \widetilde{W}^s(w)$ consists of a point $[z, w]$ and the map $(z, w) \mapsto [z, w]$ is continuous.
		      Since the set $\{ (x, y) \in W \times U ; \dist(x, y) \leq r \}$ is compact for sufficiently small $r > 0$, the map $(x, y) \mapsto [x, y]$ is uniformly continuous. 
		      Let us prove that $\dist(x, y) \leq r/2$ if $\widetilde{W}^u(x, \overline{\alpha}) \cap \widetilde{W}^s(y, \overline{\alpha}) \neq \emptyset$.
		      Take $\overline{\alpha} < r/(4C_0)$.
		      Then, as in \eqref{eq.distxy}, we can prove that $\dist(x, y) \leq 4C_0\overline{\alpha} < r/2$ if $\widetilde{W}^u(x, \overline{\alpha}) \cap \widetilde{W}^s(y, \overline{\alpha}) \neq \emptyset$.
		      Now the conclusion follows from the uniform continuity of the map $(x, y) \mapsto [x, y]$.
		\item[(3)]
		      If $z, w \in U$ satisfy
		      \[
			      z \in \widetilde{W}^s(w, \overline{\alpha}),
		      \]
		      then \eqref{eq.constofequivmetric} implies $\dist(z, w) < C_0 \overline{\alpha}$.
		      Thus, if $\overline{\alpha} > 0$ and a neighborhood $W$ of $\Lambda$ has been chosen sufficiently small, then
		      \begin{equation}
			      w \in U \cap f^{-1}(U)\label{eq.locstableisinside}
		      \end{equation}
		      when $z \in W$ and $w \in U$ satisfy $z \in \widetilde{W}^s(w, \overline{\alpha})$.
		      Making $W$ smaller if necessary, we may also assume that if $z \in W$ satisfies $f^k(z) \in W$ for all $k \geq 0$, then $z \in W^s(p, \overline{\alpha})$ for some $p \in \Lambda$.
		      
		      Let $x$, $y$, $q$ be the point in the statement.
		      Assume to the contrary that there exists $k_0 \geq 0$ such that
		      \begin{equation}
			      f^k(y) \in U, \quad 0 \leq k \leq k_0,\label{eq.defofk0}
		      \end{equation}
		      and $f^{k_0+1}(y) \notin U$.
		      It follows from \eqref{eq.defofk0} that $f^{k_0}(x) \in \widetilde{W}^s(f^k(y), \lambda^{k_0}\overline{\alpha})$.
		      By the choice of $\overline{\alpha}$ and the fact that $f^{k_0}(x) \in W$,
		      we may apply \eqref{eq.locstableisinside} with $(z, w)$ replaced by $(f^{k_0}(x), f^{k_0}(y))$ to obtain
		      \[
			      f^{k_0}(y) \in U \cap f^{-1}(U).
		      \]
		      This contradicts $f^{k_0+1}(y) \notin U$,
		      so we deduce that
		      \[
			      f^k(y) \in U, \quad k \geq 0.
		      \]
		      Then we have
		      $f^n(x) \in \widetilde{W}^s(f^n(y), \lambda^n\overline{\alpha})$ for all $n \geq 0$.
		      Since $x \in W^s(q, \overline{\alpha})$, we see that
		      \[
			      \dist(f^n(p), f^n(x)), \dist(f^n(x), f^n(y)) \to 0
		      \]
		      as $n \to \infty$
		      and therefore $y \in W^s(p, C)$ for some $C > 0$, proving (3).
	\end{enumerate}
\end{proof}
If $f$ is an Axiom A diffeomorphism,
then $\Omega(f)$ can be decomposed into
\[
	\Omega(f) = \Lambda_1 \sqcup \Lambda_2 \sqcup \cdots \sqcup \Lambda_s,
\]
where $\Lambda_1$, $\Lambda_2$, $\ldots$, $\Lambda_s$ are disjoint, compact, $f$-invariant locally maximal hyperbolic sets.

In what follows, we suppose that disjoint neighborhoods $U_i$ of $\Lambda_i$ for all $i = 1, \ldots, s$ satisfy above-mentioned properties of $U$ as $U_i=U$ and $\Lambda_i= \Lambda$.
In addition, neighborhoods $W_i$, $i = 1, 2, \ldots s$, and the numbers $\alpha, \beta$ are chosen so that Claim \ref{claim.condonWi} hold as $W_i = W$.

Let $d$ be a non-negative integer and $D_i$, $i = 1, 2$, be a compact set homeomorphic to a $d$-dimensional closed ball.
Then a continuous map $h : D_1 \to D_2$ is said to be {\it $d$-nontrivial} when one of the following properties holds:
\begin{itemize}[leftmargin=0.4cm, itemindent=0cm ]
	\item $d = 0$;
	\item $d > 0$ and $h(\partial D_1) \subset \partial D_2$, and the induced homomorphism of the reduced homology groups $(h \vert_{\partial D_1})_* :\widetilde{H}_{d-1}(\partial D_1) \to \widetilde{H}_{d-1}(\partial D_2)$ is nontrivial.
\end{itemize}

We denote by $D^k$ the closed unit ball in $\R^k$.
\begin{claim}\label{claim.dnontrivialissurjective}
	If $h : D_1 \to D_2$ is $d$-nontrivial, then $h$ is surjective.
\end{claim}
\begin{proof}
	When $d = 0$, there is nothing to prove.
	Consider the case where $d > 0$.
	Since the induced homomorphism $(h \vert_{\partial D_1})_* : \widetilde{H}_{d-1}(\partial D_1) \to \widetilde{H}_{d-1}(\partial D_2)$ is nontrivial, we have $h(\partial D_1) = \partial D_2$.
	
	Assume to the contrary that $h$ is not surjective.
	Then there is $x_0 \in D_2 \setminus h(D_1) \subset {\rm Int}D_2$.
	Let $r : D_2 \setminus \{x_0\} \to \partial D_2$ be a retraction.
	Define a continuous map $\hat{h} : D_1 \to \partial D_2$ by $\hat{h} = r \circ h$.
	Since $D_1$ is contractive, $\hat{h}_* : \widetilde{H}_{d-1}(D_1) \to \widetilde{H}_{d-1}(\partial D_2)$ is trivial.
	If $i : \partial D_1 \to D_1$ denotes the inclusion, then we have
	\[
		(h \vert_{\partial D_1})_*  = (\hat{h}\vert_{\partial D_1})_* = \hat{h}_* \circ i_*.
	\]
	This contradicts the fact that $(h \vert_{\partial D_1})_*$ is nontrivial.
\end{proof}
\begin{lem}\label{lem.nontrivial}$\,$
	\begin{enumerate}[leftmargin=0.8cm, itemindent=0cm]
		\item[$(1)$] If $h : D_1 \to D_2$ and $g : D_2 \to D_3$ are $d$-nontrivial, then $g \circ h : D_1 \to D_3$ is a $d$-nontrivial map.
		\item[$(2)$] If $h : D_1 \to D_2$ and $g : D_3 \to D_4$ are $d$-nontrivial and $e$-nontrivial, respectively, then $h \times g : D_1 \times D_3 \to D_2 \times D_4$ is a $(d+e)$-nontrivial map.
	\end{enumerate}
\end{lem}
\begin{proof}
	\begin{enumerate}[leftmargin=*, itemindent=0.6cm, itemsep=0.2cm]
		\item When $d = 0$, there is nothing to prove.
		      Consider the case where $d > 0$.
		      Then
		      $g \circ h(\partial D_1) \subset g(\partial D_2) \subset \partial D_3$. Since $\widetilde{H}_{d-1}(\partial D_i) \cong \Z$ for $i = 1, 2, 3$, the fact that $(h \vert_{\partial D_1})_*$ and $(g \vert_{\partial D_2})_*$ are nontrivial is equivalent to the injectivity of these maps.
		      Thus, $(g \circ h \vert_{\partial D_1})_* = (g \vert_{\partial D_2})_* \circ (h \vert_{\partial D_1})_*$ is injective, and hence nontrivial.
		\item When $d = 0$, then we may identify $D_1 \times D_3$ with $D_3$, and $D_2 \times D_4$ with $D_4$, respectively. Thus, the map $g \times h$ can be identified with $h$, and is $e$-nontrivial. Similar argument holds for the case $e = 0$.
		      
		      Assume that $e, d > 0$. We may regard $D_1 = D_2 = D^d$ and $D_3 = D_4 = D^e$.
		      From (1) and the fact that $g \times h = (g \times {\rm id}_{D^e}) \circ ({\rm id}_{D^d} \times h)$, it is enough to prove that $g \times {\rm id}_{D^e}$ and ${\rm id}_{D^d} \times h$ are $(d+e)$-nontrivial.
		      Let us consider the former case.
		      Denote by $F : \partial (D^d \times D^e) \to  \partial (D^d \times D^e)$ the map $g \times {\rm id}_{D^e} \vert_{ \partial (D^d \times D^e)}$.
		      Let
		      \[
			      A = \partial D^d \times {\rm Int}D^e
		      \]
		      and
		      \[
			      B = (D^d \times \partial D^e) \cup \left( \partial D^d \times (D^e \setminus \{0\}^e) \right).
		      \]
		      Note that $F^{-1}(A) = A$ and $F^{-1}(B) = B$.
		      Then, by considering the Mayer-Vietoris sequence for  $A, B \subset \partial (D^d \times D^e)$, we have the following commutative diagram:
		      \[
			      \xymatrix{
			      \widetilde{H}_{d+e-1}(A) \oplus \widetilde{H}_{d+e-1}(B) \ar[r] \ar[d]_{F_* \oplus F_*}
			      &
			      \widetilde{H}_{d+e-1}(\partial (D^d \times D^e)) \ar@{^{(}-_>}[r] \ar[d]_{F_*} &
			      \widetilde{H}_{d+e-2}(A \cap B) \ar[d]_{F_*} \\
			      \widetilde{H}_{d+e-1}(A) \oplus \widetilde{H}_{d+e-1}(B) \ar[r] &
			      \widetilde{H}_{d+e-1}(\partial (D^{d} \times D^{e})) \ar@{^{(}-_>}[r]  \ar@{}[lu]|{\circlearrowright}& 
			      \widetilde{H}_{d+e-2}(A \cap B) \ar@{}[lu]|{\circlearrowright}
			      }
		      \]
		      Since $A \approx \partial D^{d}$ and $B \approx \partial D^{e}$, we obtain $\widetilde{H}_{d+e-1}(A) \oplus \widetilde{H}_{d+e-1}(B) \cong 0$.
		      Thus, the maps $\widetilde{H}_{d+e-1}(\partial (D^d \times D^e)) \to \widetilde{H}_{d+e-2}(A \cap B)$ written in the diagram are injective.
		      Since $A \cap B = \partial D^d \times ({\rm Int} D^e \setminus \{0\}^e)$, we have
		      $F \vert_{A \cap B} = g \vert_{\partial D^d} \times {\rm id}_{{\rm Int} D^e \setminus \{0\}^e}$, which implies that the map $F_* : \widetilde{H}_{d+e-2}(A \cap B) \to \widetilde{H}_{d+e-2}(A \cap B)$ is nontrivial (see \cite[Corollary 3B.7]{HatcherAlgTop}).
		      Thus, 
		      \[
			      \widetilde{H}_{d+e-1}(\partial (D^d \times D^e)) \to \widetilde{H}_{d+e-2}(A \cap B) \to \widetilde{H}_{d+e-2}(A \cap B)
		      \]
		      is injective.
		      The commutativity shows that $F_* : \widetilde{H}_{d+e-1}(\partial (D^d \times D^e)) \to \widetilde{H}_{d+e-1}(\partial (D^d \times D^e))$ is injective, and hence nontrivial.
		      The proof for the later case is similar.
	\end{enumerate}
\end{proof}

Let $0 \leq m_0 \leq m$ and take $x \in U_i$ to satisfy $u = \dim \widetilde{W}^u(x) \leq m_0$.
For sufficiently small $\alpha, \beta > 0$, consider the following properties for pairs of an embedded $m_0$-dimensional closed ball $D \subset M$ and a continuous map $h \times \eta : D \to \widetilde{W}^u(x, \beta) \times D^{m_0-u}$;
\begin{enumerate}
	\item[{\bf(P1)}] $h \times \eta : D \to \widetilde{W}^u(x, \beta) \times D^{m_0-u}$ is $m_0$-nontrivial (note that $\widetilde{W}^u(x, \beta) \times D^{m_0-u}$ is homeomorphic to an $m_0$-dimensional closed ball).
	\item[{\bf(P2)}] For all $y \in (h \times \eta)^{-1} \bigl( {\rm Int} (\widetilde{W}^u(x, \beta) \times D^{m_0-u}) \bigr)$, we have
	      \[
		      h(y) = [x, y] \in \widetilde{W}^s(y, \alpha) \cap \widetilde{W}^u(x, \beta).
	      \]
\end{enumerate}
\begin{dfn}
	Denote by $\D(\alpha, \widetilde{W}^u(x, \beta) \times D^{m_0-u})$ the set consisting of pairs of an embedded $m_0$-dimensional closed ball $D \subset M$ and a continuous map $h \times \eta : D \to \widetilde{W}^u(x, \beta) \times D^{m_0-u}$ satisfying {\bf(P1)} and {\bf(P2)}.
\end{dfn}
\begin{dfn}
	For $(D, h \times \eta) \in \D(\alpha, \widetilde{W}^u(x, \beta) \times D^{m_0-u})$, we set
	\[
		V^u(h \times \eta)
		= (h \times \eta)^{-1} \bigl( {\rm Int} (\widetilde{W}^u(x, \beta) \times D^{m_0-u}) \bigr).
	\]
\end{dfn}
Note that $D$ is a disjoint union of
\[
	(h \times \eta)^{-1} \bigl( {\rm Int} (\widetilde{W}^u(x, \beta) \times D^{m_0-u}) \bigr) \text{ and } (h \times \eta)^{-1} \bigl( \partial (\widetilde{W}^u(x, \beta) \times D^{m_0-u}) \bigr).
\]
This implies that we have either $(h \times \eta)(y) \in \partial (\widetilde{W}^u(x, \beta) \times D^{m_0-u})$ or $h(y) = [x, y]$ for all $y \in D$.
When it is not necessary to specify the continuous map $h \times \eta$, we just write $D \in \D(\alpha, \widetilde{W}^u(x, \beta) \times D^{m_0-u})$ instead of $(D, h \times \eta) \in \D(\alpha, \widetilde{W}^u(x, \beta) \times D^{m_0-u})$.
\begin{figure}[h]
	\includegraphics[width=12cm]{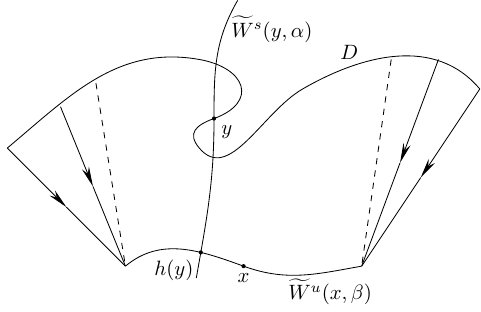}
	\caption{$(D, h \times \eta) \in \D(\alpha, \widetilde{W}^u(x, \beta) \times D^{m_0-u})$ when $m_0-u=0$.}\label{figure.Dform0u0}
\end{figure}

\begin{figure}[h]
	\centering
	\begin{minipage}[b]{1\columnwidth}
		\centering
		\includegraphics[width=10cm]{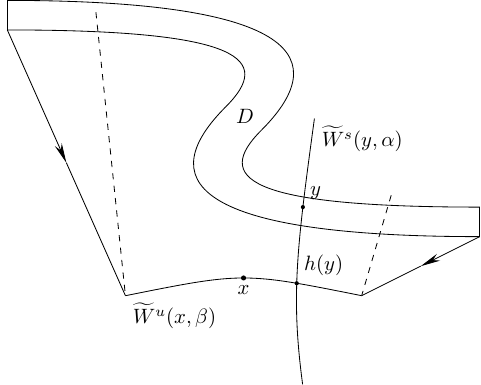}
		\subcaption{$h : D \to \widetilde{W}^u(x, \beta)$}\label{figure.Dforh}
		\vfill
		\centering
		\includegraphics[width=12cm]{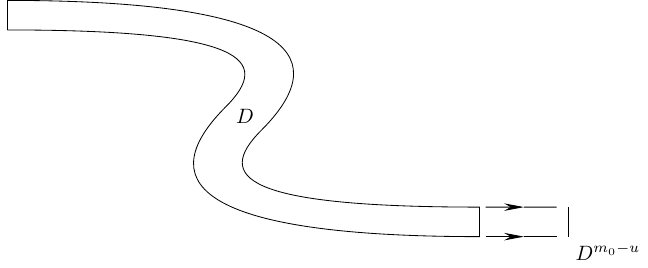}
		\subcaption{$\eta : D \to D^{m_0-u}$}\label{figure.Dforeta}
	\end{minipage}
	\caption{$(D, h \times \eta) \in \D(\alpha, \widetilde{W}^u(x, \beta) \times D^{m_0-u})$ when $m_0-u>0$.}\label{figure.Dform0upositive}
\end{figure}
\begin{rk}\label{rk.motivationofD(a(w(xb)))}
	Let $(D, h \times \eta) \in \D(\alpha, \widetilde{W}^u(x, \beta) \times D^{m_0-u})$. Then, by continuity of $h$ and {\bf (P2)}, we have
	\[
		h(y) = [x, y] \in \widetilde{W}^s(y, \alpha) \cap \widetilde{W}^u(x, \beta)
	\]
	for all $y \in \overline{V^u(h \times \eta)}$, which is
	\[
		\overline{(h \times \eta)^{-1} \bigl( {\rm Int} (\widetilde{W}^u(x, \beta) \times D^{m_0-u}) \bigr)}.
	\]
	It follows from Claim \ref{claim.dnontrivialissurjective} that the map $h \times \eta$ is surjective.
	Thus, there exists $y \in D$ such that $h(y) = z \in \widetilde{W}^s(y, \alpha)$ for all $z \in \widetilde{W}^u(x, \beta)$, that is, every point in $\widetilde{W}^u(x, \beta)$ is ``$\alpha$-close'' to some point in $D$ (see Figures \ref{figure.Dform0u0} and \ref{figure.Dform0upositive}).
\end{rk}

\section{Homological inclination lemma}\label{sec.homologicalinclinationlemma}
It was proved by Morimoto \cite{morimoto1979method}, Sawada \cite{sawada1980extended} and Robinson \cite{R} that if an Axiom A diffeomorphism $f$ satisfies the strong transversality condition, then $f$ has the shadowing property.
Petrov and Pilyugin \cite{PETROV2015} proved that an Axiom A diffeomorphism satisfying the $C^0$-transversality condition on a closed surface also satisfies the shadowing property, using a strategy similar to \cite{R}. In this paper, we follow that approach and prove Theorem \ref{thm.homologyshadowing} by using the homological version of the inclination lemma.
In this section, 
we give several lemmas to prove the homological inclination lemma.
The classical inclination lemma asserts that if a disk $D$ is transverse to a stable manifold of a hyperbolic fixed point, then $f^n(D)$ is $C^1$-close to a compact part of its unstable manifold for sufficiently large $n$.
Here, we consider the $T$-condition instead of the transversality, and $D \in \D(\alpha, \widetilde{W}^u(x, \beta) \times D^k)$ instead of a transversal disk $D$.

In the following lemma, we prove that if a stable semi-invariant disk and an embedded closed ball $D$ satisfy the $T$-condition, then $D \in \D(\alpha, \widetilde{W}^u(x, \beta) \times D^k)$.
\begin{lem}\label{lem.C0transstarimplyeclosed}
	Let $x \in W_i$, $1 \leq i \leq s$, satisfy $f^k(x) \in W_i$ for all $k \geq 0$,
	and let $u = \dim\widetilde{W}^u(x)$. For $\overline{\alpha}$ given in Claim \ref{claim.condonWi}, let $\alpha \in (0, \overline{\alpha})$.
	Suppose that 
	$\widetilde{W}^s(x, \alpha/2)$ and an embedded closed ball $W \subset M$ with dimension $d$ satisfy the $T$-condition at some point $p$.
	Then, for all $0 \leq k \leq d-u$,
	there exists $\beta_1 > 0$ and an embedded $(u+k)$-dimensional closed ball $D \subset W$ centered at $p$ and a continuous map $h \times \eta : D \to \widetilde{W}^u(x, \beta_1) \times D^{k}$ such that $(D, h \times \eta) \in \D(\alpha, \widetilde{W}^u(x, \beta_1) \times D^{k})$.
\end{lem}
\begin{proof}
	When $u = 0$, for $0 \leq k \leq d$ and an embedded $k$-dimensional closed ball $D \subset W \cap \widetilde{W}^s(x, \alpha)$ centered at $p$,
	take a constant map $h : D \to \{x\}$ and any homeomorphism $\eta : D \to D^{k}$. Then we have $(D, h \times \eta) \in \D(\alpha, \widetilde{W}^u(x, \beta_1) \times D^{k})$.
	
	When $u > 0$,
	use Claim \ref{claim.condonWi} (3) to have $z \in \Lambda_i$ with $x \in W^s(z, \overline{\alpha})$.
	Let $C > 0$ be the constant given in Claim \ref{claim.condonWi}.
	Since $W^s(z, C)$ is an embedded closed ball, 
	there exist a neighborhood $U \subset \B(x, \alpha, \beta)$ of $p$ and a diffeomorphism $\phi : U \to (-1, 1)^m$ such that
	\begin{equation}
		\phi(W^s(z, C) \cap U) = (-1, 1)^{m-u} \times \{0\}^u.\label{eq.defofU}
	\end{equation}
	By the $T$-condition, there exists an embedded $u$-dimensional closed ball $D_0 \subset W$ centered at $p$ satisfying the second item of Definition \ref{dfn.TUaUbcondition} such that
	\[
		\widetilde{W}^u(x, \beta/2) \cap \widetilde{W}^s(y, 3\alpha/4) = \{[x, y]\} \neq \emptyset, \quad y \in D_0.
	\]
	Since $D_0$ is an embedded closed ball, there is a homeomorphism $\tau : D_0 \times D^k \to U$ onto its image such that $\tau(y, 0) = y$ for all $y \in D_0$ and
	\[
		\widetilde{W}^u(x, \beta) \cap \widetilde{W}^s(y, \alpha) = \{[x, y]\} \neq \emptyset, \quad y \in \tau(D_0 \times D^k).
	\]
	
	Let $f_0 : U \setminus W^s(z, C) \to \widetilde{W}^u(x, \beta)$ be a continuous map satisfying $f_0(y) = [x, y]$.
	Then we have $f_0(y) \neq x$ for all $y \in U \setminus W^s(z, C)$.
	In fact, if $y \in U$ satisfies $f_0(y) = x$, then $x \in \widetilde{W}^s(y, \alpha)$.
	By Claim \ref{claim.condonWi} (3), $y \in \widetilde{W}^s(z, C)$.
	Therefore, there exist $t_0 \in (0, 1)$ and $\beta_1 \in (0, \beta)$ such that
	\begin{equation}
		f_0(\tau(\partial D_0 \times D^k(t_0))) \cap \widetilde{W}^u(x, \beta_1) = \emptyset,\label{eq.conditiononf0}
	\end{equation}
	where $D^k(t_0) \subset \R^k$ is a closed ball with radius $t_0$. Let $r : \widetilde{W}^u(x, \beta_0) \to \widetilde{W}^u(x, \beta_1)$ be a $u$-nontrivial retraction such that
	\begin{equation}
		r(\widetilde{W}^u(x, \beta_0) \setminus \widetilde{W}^u(x, \beta_1)) = \partial \widetilde{W}^u(x, \beta_1).\label{eq.rrDpartial}
	\end{equation}
	Let $\hat{D} = D_0 \times D^k(t_0)$ and define $\hat{h} \times \hat{\eta} : \hat{D} \to \widetilde{W}^u(x, \beta_1) \times D^k$ by
	\[
		\hat{h} \times \hat{\eta}(y, v) = (r \circ f_0\circ \tau(y, v), 1/t_0 \cdot v).
	\]
	Using \eqref{eq.rrDpartial} and \eqref{eq.conditiononf0}, we get
	\begin{align}
		\hat{h} \times \hat{\eta}(\partial \hat{D})
		 & = \hat{h} \times \hat{\eta}(\partial D_0 \times D^k(t_0))
		\cup \hat{h} \times \hat{\eta}(D_0 \times \partial D^k(t_0))                           \\
		 & \subset \partial \widetilde{W}^u(x, \beta_1) \times D^k
		\cup \widetilde{W}^u(x, \beta_1) \times \partial D^k                                   \\
		 & \subset \partial (\widetilde{W}^u(x, \beta_1) \times D^k).\label{eq.inclusionfor31}
	\end{align}
	In addition, $\hat{h} \times \hat{\eta} \vert_{\partial \hat{D}}$ is homotopic to
	\[
		\partial \hat{D} \ni (y, v) \mapsto (r \circ f_0\circ \tau(y, 0), 1/t_0 \cdot v) \in \partial (\widetilde{W}^u(x, \beta_1) \times D^k).
	\]
	By the choice of $r$ and the fact that $f_0 : U \setminus W^s(z, C) \to \widetilde{W}^u(x, \beta)$ is a retraction, $r \circ f_0 \circ \tau(\cdot, 0) \vert_{\partial D_0}$ induces a nontrivial homomorphism
	\[
		(r \circ f_0 \circ \tau(\cdot, 0) \vert_{\partial D_0})_* : \widetilde{H}_{u-1}(\partial D_0) \to \widetilde{H}_{u-1}(\partial \widetilde{W}^u(x, \beta_1))
	\]
	between reduced homology groups.
	Thus, Lemma \ref{lem.nontrivial} (2) implies that $\hat{h} \times \hat{\eta}$ is $(u+k)$-nontrivial.
	Let $D = \tau(\hat{D})$ and define $h \times \eta = (\hat{h} \times \hat{\eta}) \circ \tau^{-1} : D \to \widetilde{W}^u(x, \beta_1) \times D^k$.
	Then, let us prove that $(D, h \times \eta) \in \D(\alpha, \widetilde{W}^u(x, \beta_1) \times D^k)$.
	Since $\hat{h} \times \hat{\eta}$ is $(u+k)$-nontrivial, $h \times \eta$ is also $(u+k)$-nontrivial.
	For $z \in (h \times \eta)^{-1}({\rm Int}(\widetilde{W}^u(x, \beta_1) \times D^k))$,
	since $h(z) = r \circ f_0 \circ \tau \circ \tau^{-1}(z) = r \circ f_0(z) \in {\rm Int}\widetilde{W}^u(x, \beta_1)$,
	using \eqref{eq.rrDpartial}, we obtain
	\[
		h(z) = r \circ f_0(z) = f_0(z) = [x, z].
	\]
\end{proof}

\begin{lem}\label{lem.basicofepsilonclosever2}
	Let $0 \leq u \leq m_0 \leq m$.
	Given $x \in U_i$, $1 \leq i \leq s$, with $\dim \widetilde{W}^u(x) = u$
	and $(D, h \times \eta) \in \D(\alpha, \widetilde{W}^u(x, \beta) \times D^{m_0-u})$,
	if $\beta' \in (0, \beta)$, then there exists $g : D \to \widetilde{W}^u(x, \beta')$ satisfying the following properties:
	\begin{enumerate}[leftmargin=0.8cm, itemindent=0cm]
		\item[$(1)$] $(D, g \times \eta) \in \D(\alpha, \widetilde{W}^u(x, \beta') \times D^{m_0-u})$;
		\item[$(2)$] $V^u(g \times \eta) \subset V^u(h \times \eta)$.
	\end{enumerate}
\end{lem}
\begin{proof}
	Let $r : \widetilde{W}^u(x, \beta) \to \widetilde{W}^u(x, \beta')$ be a $u$-nontrivial retraction such that
	\begin{equation}
		r(\widetilde{W}^u(x, \beta) \setminus \widetilde{W}^u(x, \beta')) \subset \partial \widetilde{W}^u(x, \beta').\label{eq.condonretract}
	\end{equation}
	Let $h \times \eta : D \to \widetilde{W}^u(x, \beta) \times D^{m_0-u}$ be a continuous map with $(D, h \times \eta) \in \D(\alpha, \widetilde{W}^u(x, \beta) \times D^{m_0-u})$. 
	Let us prove that $(r \circ h) \times \eta : D \to \widetilde{W}^u(x, \beta') \times D^{m_0-u}$ satisfies conditions {\bf (P1)} and {\bf (P2)} above (i.e., $(D, (r \circ h) \times \eta) \in \D(\alpha, \widetilde{W}^u(x, \beta') \times D^{m_0-u})$).
	Since $h \times \eta$ satisfies {\bf (P1)} and $r$ is $u$-nontrivial, Lemma \ref{lem.nontrivial} implies that the map $(r \circ h) \times \eta = (r \times {\rm id}_{D^{m_0-u}}) \circ (h \times \eta)$ satisfies {\bf (P1)}.
	By \eqref{eq.condonretract}, we have
	\begin{align}
		 & ((r \circ h) \times \eta)^{-1}\bigl({\rm Int} (\widetilde{W}^u(x, \beta') \times D^{m_0-u})\bigr)                                    \\
		 & = (h \times \eta)^{-1} \circ (r \times {\rm id}_{D^{m_0-u}})^{-1}\bigl({\rm Int} (\widetilde{W}^u(x, \beta') \times D^{m_0-u})\bigr) \\
		 & = (h \times \eta)^{-1}({\rm Int} (\widetilde{W}^u(x, \beta') \times D^{m_0-u})).\label{eq.incluofrheta}
	\end{align}
	Thus, by applying {\bf (P2)} for $h \times \eta$, we have
	\[
		r \circ h(y) = h(y) = [x, y] \in \widetilde{W}^s(y, \alpha) \cap \widetilde{W}^u(x, \beta')
	\]
	for all $y \in ((r \circ h) \times \eta)^{-1}\bigl({\rm Int} (\widetilde{W}^u(x, \beta') \times D^{m_0-u})\bigr)$.
	Thus, we see that $(D, (r \circ h) \times \eta) \in \D(\alpha, \widetilde{W}^u(x, \beta') \times D^{m_0-u})$.
	Moreover, it follows from \eqref{eq.incluofrheta} that
	\begin{align}
		V^u((r \circ h) \times \eta)
		=       & ((r \circ h) \times \eta)^{-1}({\rm Int}(\widetilde{W}^u(x, \beta') \times D^{m_0-u}))             \\
		=       & (h \times \eta)^{-1}({\rm Int} (\widetilde{W}^u(x, \beta') \times D^{m_0-u}))                      \\
		\subset & (h \times \eta)^{-1}({\rm Int}(\widetilde{W}^u(x, \beta) \times D^{m_0-u})) = V^u(h \times \eta).
	\end{align}
\end{proof}

In the following lemma, we prove that if $D$ is ``$\alpha$-close'' to $\widetilde{W}^u(x, \beta)$ in the sense of Remark \ref{rk.motivationofD(a(w(xb)))}, then $f(D)$ is ``$\lambda\alpha$-close'' to $\widetilde{W}^u(f(x), \beta/\lambda)$.
Combining this lemma and Lemma \ref{lem.C0transstarimplyeclosed}, we obtain a homology version of the inclination lemma.
\begin{lem}\label{lem.closeddiscchange}
	Let $0 \leq u \leq m_0 \leq m$.
	Suppose $x \in W_i$, $1 \leq i \leq s$, satisfies $\dim \widetilde{W}^u(x) = u$ and $(D, h \times \eta) \in \D (\alpha, \widetilde{W}^u(x, \beta) \times D^{m_0-u})$.
	Then there exists a continuous map $g \times \theta : f(D) \to \widetilde{W}^u(f(x), \beta/\lambda) \times D^{m_0-u}$ satisfying the following properties:
	\begin{enumerate}[leftmargin=0.8cm, itemindent=0cm]
		\item[$(a)$] $(f(D), g \times \theta) \in \D(\lambda\alpha, \widetilde{W}^u(f(x), \beta/\lambda) \times D^{m_0-u})$;
		\item[$(b)$] $V^u(g \times \theta) \subset f (V^u(h \times \eta))$.
	\end{enumerate}
\end{lem}
\begin{proof}
	Let $r : f(\widetilde{W}^u(x, \beta)) \to \widetilde{W}^u(f(x), \beta/\lambda)$ be a $u$-nontrivial retraction such that $r(f(\widetilde{W}^u(x, \beta)) \setminus \widetilde{W}^u(f(x), \beta/\lambda)) = \partial \widetilde{W}^u(f(x), \beta/\lambda)$. 
	Consider the maps
	\[
		g = r \circ f \circ h \circ f^{-1}\vert_{f(D)} : f(D) \to \widetilde{W}^u(f(x), \beta/\lambda)
	\]
	and
	\[
		\theta = \eta \circ f^{-1} : f(D) \to D^{m_0-u}.
	\]
	For the proof of property (a),
	note that
	\begin{align}
		 & (g \times \theta)^{-1}({\rm Int}(\widetilde{W}^u(f(x), \beta/\lambda) \times D^{m_0-u}))                                                  \\
		 & = f \circ h^{-1} \circ f^{-1} \circ r^{-1} ({\rm Int} \widetilde{W}^u(f(x), \beta/\lambda)) \cap \theta^{-1}( {\rm Int}D^{m_0-u})         \\
		 & = f \circ h^{-1}({\rm Int} f^{-1}(\widetilde{W}^u(f(x), \beta/\lambda))) \cap f \circ \eta^{-1}({\rm Int}D^{m_0-u})                       \\
		 & = f(h^{-1}({\rm Int} f^{-1}(\widetilde{W}^u(f(x), \beta/\lambda))) \cap \eta^{-1}({\rm Int}D^{m_0-u}))                                    \\
		 & = f \circ (h \times \eta)^{-1}({\rm Int} (f^{-1}(\widetilde{W}^u(f(x), \beta/\lambda)) \times D^{m_0-u})).\label{eq.inclucloseddischange}
	\end{align}
	Thus, for $y \in (g \times \theta)^{-1}( {\rm Int}(\widetilde{W}^u(f(x), \beta/\lambda) \times D^{m_0-u}))$, we have
	\[
		(h \times \eta) \circ f^{-1}(y) \in {\rm Int} (f^{-1}(\widetilde{W}^u(f(x), \beta/\lambda)) \times D^{m_0-u}).
	\]
	Since $(D, h \times \eta) \in \D(\alpha, \widetilde{W}^u(x, \beta) \times D^{m_0-u})$ satisfies {\bf (P2)}, we have
	\[
		h(f^{-1}(y)) = [x, f^{-1}(y)] \in {\rm Int} f^{-1}(\widetilde{W}^u(f(x), \beta/\lambda))
	\]
	for all $y \in (g \times \theta)^{-1}( {\rm Int}\widetilde{W}^u(f(x), \beta/\lambda) \times D^{m_0-u})$.
	This implies that
	\begin{align}
		g(y)
		 & = r(f(h(f^{-1}(y)))) = r(f ( [x, f^{-1}(y)] ))                                                              \\
		 & = f ( [x, f^{-1}(y)] ) \in \widetilde{W}^u(f(x), \beta/\lambda) \cap f(\widetilde{W}^s(f^{-1}(y), \alpha))
	\end{align}
	for all $y \in (g \times \theta)^{-1}( {\rm Int}\widetilde{W}^u(f(x), \beta/\lambda) \times D^{m_0-u})$.
	Thus, 
	\[
		d^s_y(y, g(y)) \leq \lambda d^s_{f^{-1}(y)}(f^{-1}(y), f^{-1}(g(y))) \leq \lambda \alpha,
	\]
	proving property (a).
	This together with \eqref{eq.inclucloseddischange} yields the conclusion.
\end{proof}

Let us define the following new property;
\begin{enumerate}
	\item[{\bf (P2')}] For all $y \in (\overline{h} \times \overline{\eta})^{-1}({\rm Int} (\widetilde{W}^u(x, \beta-\e/2) \times D^{m_0-u}))$,
	      \[
		      \widetilde{W}^s(y, \alpha) \cap \widetilde{W}^u(x) = \{ [x, y] \} \neq \emptyset
	      \]
	      and $d^u_{x}(\overline{h}(y), [x, y]) < \e/(2m)$.
\end{enumerate}
The following lemma asserts that if $\overline{h} \times \overline{\eta}$ satisfies {\bf (P1)} and {\bf (P2')}, then there exists a map $g \times \theta$ satisfying {\bf (P1)} and {\bf (P2)}.
Since {\bf (P2')} is weaker than {\bf (P2)}, we have to make $\beta$ smaller to have such $g \times \theta$.
\begin{lem}\label{lem.generalizedpurt}
	Let $\alpha, \beta > 0$ and $\e \in (0, \beta)$ be given.
	Suppose that there are $x \in W_i$, $1 \leq i \leq s$, with $\dim\widetilde{W}^u(x) = u$, an embedded $m_0$-dimensional closed ball $D$ and a continuous map $\overline{h} \times \overline{\eta} : D \to \widetilde{W}^u(x, \beta) \times D^{m_0-u}$ for some $0 \leq u \leq m_0 \leq m$ satisfying {\bf (P1)} and {\bf (P2')}.
	Then there exists a continuous map $g \times \theta : D \to \widetilde{W}^u(x, \beta - \e) \times D^{m_0-u}$ satisfying the following properties:
	\begin{enumerate}[leftmargin=0.8cm, itemindent=0cm]
		\item[$(a)$]
		      $(D, g \times \theta) \in \D(\alpha, \widetilde{W}^u(x, \beta - \e) \times D^{m_0-u})$;
		\item[$(b)$]
		      $V^u(g \times \theta)
			      \subset (\overline{h} \times \overline{\eta})^{-1}({\rm Int}(\widetilde{W}^u(x, \beta) \times D^{m_0-u}))$.
	\end{enumerate}
\end{lem}
\begin{proof}
	Denote
	\[
		\widetilde{E}^u_x(\beta) = \{ y \in \widetilde{E}^u_x ; \lvert y \rvert \leq \beta \},
	\]
	where $\widetilde{E}^u_x$ is given in the beginning of Section \ref{sec.prepforthm.homologyshadowing}.
	Let
	\[
	K_0 = (\overline{h} \times \overline{\eta})^{-1}(\widetilde{W}^u(x, \beta - \e/2) \times D^{m_0-u}(1/2))\]
	and
	\[
		K_1 = (\overline{h} \times \overline{\eta})^{-1}(\partial (\widetilde{W}^u(x, \beta) \times D^{m_0-u})),
	\]
	where $D^{m_0-u}(1/2)$ is a closed ball in $\R^{m_0-u}$ with radius $1/2$.
	Denote by $\pr_{\widetilde{E}^u_x} : \widetilde{E}^s_x \oplus \widetilde{E}^u_x \to \widetilde{E}^u_x$ the projection to the second component.
	Then, define a diffeomorphism $P : \widetilde{W}^u(x, \beta) \to \widetilde{E}^s_x(\beta)$ by $P(z) = \pr_{\widetilde{E}^u_x} \circ \exp_{x}^{-1}(z)$.
	By the definition of $d^u_x$, we have
	\begin{equation}
		d^u_x(y, z) = \lvert P(y) - P(z) \rvert, \quad y, z \in \widetilde{W}^u(x, \beta).\label{eq.defofduy}
	\end{equation}
	Let $v_1, v_2, \ldots, v_u \in \widetilde{E}^u_x$ be orthonormal basis of $\widetilde{E}^u_x$.
	We need the following claim.
	\begin{claim}\label{claim.3}
		There exists a continuous map $H : D \to \widetilde{E}^u_x(\e/2)$ such that
		\[
			H(y) =
			\begin{cases}
				P([x, y]) - P(\overline{h}(y)), & y \in K_0,  \\
				0,                              & y \in K_1,
			\end{cases}
		\]
		and $\lvert H(y) \rvert \leq \e/2$ for all $y \in D$.
	\end{claim}
	For the proof, by {\bf (P2')},
	\[
		\widetilde{W}^s(y, \alpha) \cap \widetilde{W}^u(x, \beta) = \{ [x, y] \} \neq \emptyset
	\]
	for all $y \in K_0$.
	Define $H_0 : K_0 \sqcup K_1 \to \widetilde{E}^u_x$ by
	\[
		H_0(y) =
		\begin{cases}
			P([x, y]) - P(\overline{h}(y)), & y \in K_0,  \\
			0,                              & y \in K_1.
		\end{cases}
	\]
	For each $j = 1, 2, \ldots, u$, denote by $\pr_j : \widetilde{E}^u_x \to \R$ the projection
	\[
		\pr_j(t_1v_1 + \cdots + t_uv_u) = t_j, \quad (t_1, \ldots, t_u) \in \R^u.
	\]
	Using Tietze extension theorem for $\pr_j \circ H_0 : K_0 \sqcup K_1 \to \R$, we find a continuous map $H_j : D \to \R$ such that
	\[
		H_j(y) = \pr_j \circ H_0(y), \quad y \in K_0 \sqcup K_1,
	\]
	and
	\[
		\sup \{\lvert H_j(y) \rvert ; y \in D \}
		= \sup \{ \lvert \pr_j \circ H_0(y) \rvert ; y \in K_0 \sqcup K_1 \}
		\leq \e/(2m).
	\]
	The last inequality follows from \eqref{eq.defofduy} and the property {\bf (P2')}.
	Define $H : D \to \widetilde{E}^u_x$ by $H(y) = H_1(y)v_1 + \cdots + H_u(y)v_u$.
	Then we have
	\[
		\lvert H(y) \rvert
		\leq \lvert H_1(y) \rvert + \cdots + \lvert H_u(y) \rvert
		\leq \e/(2m) + \cdots + \e/(2m) \leq \e/2
	\]
	for all $y \in D$, which proves the claim \ref{claim.3}.
	
	Let $\tau : D \to [0, 1]$ be a continuous map defined by
	\[
		\tau(y) = 
		\begin{cases}
			(\beta - \lvert P \circ \overline{h}(y) \rvert)/(\e/2), & \beta - \e/2 \leq \lvert P \circ \overline{h}(y) \rvert \leq \beta, \\
			1,                                                      & \lvert P \circ \overline{h}(y) \rvert \leq \beta - \e/2.
		\end{cases}
	\]
	Then we see that $\tau(y) \cdot H(y) + P \circ \overline{h}(y) \in \widetilde{E}^u_x(\beta)$ for all $y \in D$.
	In fact, for all $y \in D$ with $\beta - \e/2 \leq \lvert P \circ \overline{h}(y) \rvert \leq \beta$,
	\begin{align}
		\lvert \tau(y) \cdot H(y) + P \circ \overline{h}(y) \rvert
		 & \leq \tau(y) \cdot \lvert H(y) \rvert + \lvert P \circ \overline{h}(y) \rvert                                  \\
		 & \leq (\beta - \lvert P \circ \overline{h}(y) \rvert)/(\e/2) \cdot \e/2 + \lvert P \circ \overline{h}(y) \rvert \\
		 & = \beta,
	\end{align}
	and for all $y \in D$ with $\lvert P \circ \overline{h}(y) \rvert \leq \beta - \e/2$, 
	\[
		\lvert \tau(y) \cdot H(y) + P \circ \overline{h}(y) \rvert
		\leq \lvert H(y) \rvert + \lvert P \circ \overline{h}(y) \rvert
		\leq \e/2 + (\beta - \e/2) = \beta.
	\]
	Thus, we can define a continuous map $G : D \to \widetilde{E}^u_x(\beta)$ by
	\[
		G(y) = \tau(y) \cdot H(y) + P \circ \overline{h}(y).
	\]
	
	Assume that $G(y) \in \widetilde{E}^u_x(\beta - \e)$ and $\overline{\eta}(y) \in D^{m_0-u}(1/2)$ for some $y \in D$. Since $\lvert G(y) - P \circ \overline{h}(y) \rvert = \lvert \tau(y) \cdot H(y) \rvert = \lvert H(y) \rvert < \e/2$, we have $P \circ \overline{h}(y) \in \widetilde{E}^u_x(\beta - \e/2)$.
	This and $\overline{\eta}(y) \in D^{m_0-u}(1/2)$ imply $y \in K_0$.
	Therefore,
	\begin{equation}
		y \in K_0,\label{eq.inverseiimageofG}
	\end{equation}
	\begin{equation}
		G(y) = P([x, y])\label{eq.Gofinteriorofdomain}
	\end{equation}
	for all $y \in D$ satisfying $G(y) \in \widetilde{E}^u_x(\beta - \e)$ and $\overline{\eta}(y) \in D^{m_0-u}(1/2)$.
	Let $r : \widetilde{W}^u(x, \beta) \to \widetilde{W}^u(x, \beta - \e)$ be a $u$-nontrivial retraction such that $r(\widetilde{W}^u(x, \beta) \setminus \widetilde{W}^u(x, \beta - \e)) = \partial \widetilde{W}^u(x, \beta - \e)$, and let $r_\eta : D^{m_0-u} \to D^{m_0-u}$ be a continuous map defined by
	\[
		r_\eta(v) =
		\begin{cases}
			2v,               & \lvert v \rvert \leq 1/2,         \\
			v/\lvert v\rvert, & 1/2 \leq \lvert v \rvert \leq 1.
		\end{cases}
	\]
	Now let us prove that $(D, (r \circ P^{-1} \circ G) \times (r_\eta \circ \overline{\eta})) \in \D(\alpha, \widetilde{W}^u(x, \beta - \e) \times D^{m_0-u})$.
	From Lemma \ref{lem.nontrivial}, it follows that
	\[
		(r \circ \overline{h}) \times (r_\eta \circ \overline{\eta})
		= (r \times r_\eta) \circ (\overline{h} \times \overline{\eta})
	\]
	is $m_0$-nontrivial.
	This and the fact that $r \circ P^{-1} \circ G \vert_{\partial D} = r \circ \overline{h} \vert_{\partial D}$ (notice that $\partial D \subset K_0$) imply that $(r \circ P^{-1} \circ G) \times (r_\eta \circ \overline{\eta})$ is $m_0$-nontrivial.
	Since
	\begin{align}
		(r \circ P^{-1} \circ G)^{-1}({\rm Int} \widetilde{W}^u(x, \beta - \e))
		 & = G^{-1} \circ P \circ r^{-1}({\rm Int} \widetilde{W}^u(x, \beta - \e))       \\
		 & = G^{-1} \circ P({\rm Int} \widetilde{W}^u(x, \beta - \e))                    \\
		 & = G^{-1}({\rm Int} \widetilde{E}^u_x(\beta - \e))\label{eq.explicitiverseofg}
	\end{align}
	and
	\[
		(r_\eta \circ \overline{\eta})^{-1}({\rm Int}D^{m_0-u})
		= \overline{\eta}^{-1} \circ r_\eta^{-1}({\rm Int}D^{m_0-u})
		= \overline{\eta}^{-1} ({\rm Int}D^{m_0-u}(1/2)),
	\]
	we may apply \eqref{eq.Gofinteriorofdomain} for $y \in ((r \circ P^{-1} \circ G) \times (r_\eta \circ \overline{\eta}))^{-1}({\rm Int} (\widetilde{W}^u(x, \beta - \e) \times D^{m_0-u}))$ to obtain
	\[
		r \circ P^{-1} \circ G(y) = r([x, y]) = [x, y],
	\]
	which implies {\bf (P2)} as required.
	Combining \eqref{eq.explicitiverseofg} and \eqref{eq.inverseiimageofG}, we also get
	\begin{align}
		V^u((r \circ P^{-1} \circ G) \times (r_\eta \circ \overline{\eta}))
		 & =((r \circ P^{-1} \circ G) \times (r_\eta \circ \overline{\eta}))^{-1}({\rm Int} (\widetilde{W}^u(x, \beta - \e) \times D^{m_0-u}))   \\
		 & = (r \circ P^{-1} \circ G)^{-1}({\rm Int}\widetilde{W}^u(x, \beta - \e)) \cap (r_\eta \circ \overline{\eta})^{-1}({\rm Int}D^{m_0-u}) \\
		 & \subset G^{-1} ({\rm Int} \widetilde{E}^u_x(\beta - \e)) \cap \overline{\eta}^{-1}({\rm Int}D^{m_0-u}(1/2))                           \\
		 & \subset K_0                                                                                                                           \\
		 & \subset (\overline{h} \times \overline{\eta})^{-1}({\rm Int}(\widetilde{W}^u(x, \beta) \times D^{m_0-u})).
	\end{align}
	This complete the proof of Lemma \ref{lem.generalizedpurt}.
\end{proof}
The following claim follows from Claim \ref{claim.condonWi} (2).
\begin{claim}\label{claim.presmallpurtforclosed ball}
	Let $\alpha, \beta > 0$ and $\e > 0$ be given.
	If $x \in W_i$ and $(D_0, h_0 \times \eta_0) \in \D(\alpha, \widetilde{W}^u(x, \beta) \times D^{m_0-u})$, then there exists $\d > 0$ such that for all $y \in (h_0 \times \eta_0)^{-1}({\rm Int} (\widetilde{W}^u(x, \beta-\e/2) \times D^{m_0-u}))$ and $z \in B(y, \d)$,
	\[
		\widetilde{W}^s(z, \alpha + \e) \cap \widetilde{W}^u(x)= \{[x, z]\} \neq \emptyset \quad \text{ {\it and} } \quad
		d^u_{x}( h_0(y), [x, z]) < \e/(2m).
	\]
\end{claim}

We see that $(D_0, h_0 \times \eta_0) \in \D(\alpha, \widetilde{W}^u(x, \beta) \times D^k)$ means $D_0$ is ``$\alpha$-close'' to $\widetilde{W}^u(x, \beta)$ in the sense of Remark \ref{rk.motivationofD(a(w(xb)))}.
The following lemma shows that, if we take $D_1$ so "close" to $D_0$, then $D_1$ is ``$(\alpha + \e)$-close'' to $\widetilde{W}^u(x, \beta)$, which implies $(D_1, h \times \eta) \in \D(\alpha + \e, \widetilde{W}^u(x, \beta - \e) \times D^{k + l})$. Here, $D_0$ and $D_1$ are not required to have the same dimension. Consequently, one can handle pseudotrajectories passing neighborhoods of basic sets with different dimensions.
As in the proof of Lemma \ref{lem.generalizedpurt}, we need to make $\beta$ smaller.
\begin{lem}\label{lem.smallpurtforclosedball}
	Given $\alpha, \beta > 0$ and $\e > 0$, let
	$(D_0, h_0 \times \eta_0) \in \D(\alpha, \widetilde{W}^u(x, \beta) \times D^k)$ with $x \in W_i$ for some $k \geq 0$ and $1 \leq i \leq s$.
	Suppose that there are an embedded closed ball $D_1$ and a continuous map $h_1 \times \eta_1 : D_1 \to D_0 \times D^\ell$ for some $\ell \geq 0$ satisfying the following properties:
	\begin{itemize}[leftmargin=0.4cm, itemindent=0cm]
		\item $D_1$ is a $(u + k + \ell)$-dimensional closed ball when $u = \dim \widetilde{W}^u(x)$.
		\item $h_1 \times \eta_1$ is $(u + k + \ell)$-nontrivial, where $D_0 \times D^\ell$ is thought of as a $((u+k)+\ell)$-dimensional closed disk.
		\item $\dist(y, h_1(y)) < \d$ for all $y \in (h_1 \times \eta_1)^{-1}({\rm Int}(D_0 \times D^\ell))$, where $\d > 0$ is the constant given by Claim \ref{claim.presmallpurtforclosed ball}.
	\end{itemize}
	Then there exists a continuous map $h \times \eta : D_1 \to \widetilde{W}^u(x, \beta - \e) \times D^{k + \ell}$ satisfying the following properties:
	\begin{enumerate}[leftmargin=0.8cm, itemindent=0cm]
		\item[$(a)$] $(D_1, h \times \eta) \in \D(\alpha + \e, \widetilde{W}^u(x, \beta - \e) \times D^{k + \ell})$;
		\item[$(b)$] 
		      $V^u(h \times \eta) \subset (h_1 \times \eta_1)^{-1}({\rm Int} (D_0 \times D^\ell))$.
	\end{enumerate}
\end{lem}
\begin{proof}
	We identify $D^{k+\ell}$ with $D^k \times D^\ell$,
	and prove that
	\[
		(h_0 \circ h_1) \times ((\eta_0 \circ h_1) \times \eta_1) : D_1 \to \widetilde{W}^u(x, \beta) \times (D^k \times D^l)
	\]
	satisfies the two properties {\bf (P1)}, {\bf (P2')} of Lemma \ref{lem.generalizedpurt} with $(\alpha, \beta, \e)$ replaced by $(\alpha + \e, \beta, \e)$.
	\begin{figure}
		\centering
		\includegraphics[width=12cm]{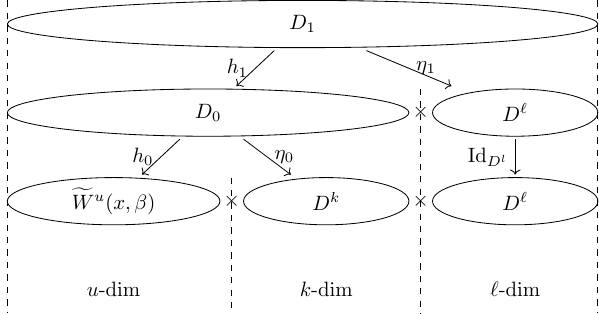}
		\caption{Construction of $(h_0 \circ h_1) \times ((\eta_0 \circ h_1) \times \eta_1)$ in Lemma \ref{lem.smallpurtforclosedball}}\label{figure.constofukl}
	\end{figure}
	By Lemma \ref{lem.nontrivial}(2), $(h_0 \times \eta_0) \times {\rm id}_{D^{\ell}}$ is $(u + k + \ell)$-nontrivial.
	Combining this and the fact that $h_1 \times \eta_1$ is $(u + k + \ell)$-nontrivial, we can use Lemma \ref{lem.nontrivial}(1) to obtain that
	\[
		(h_0 \times \eta_0 \times {\rm id}_{D^{l}}) \circ (h_1 \times \eta_1) = (h_0 \circ h_1) \times ((\eta_0 \circ h_1) \times \eta_1)
	\]
	is $(u + k + \ell)$-nontrivial (see Figure \ref{figure.constofukl}).
	Let
	\[
		z \in ((h_0 \circ h_1) \times ((\eta_0 \circ h_1) \times \eta_1))^{-1}( {\rm Int} (\widetilde{W}^u(x, \beta - \e/2) \times D^{k + l})).
	\]
	Then $h_1 \times \eta_1(z) \in {\rm Int} (D_0 \times D^\ell)$,
	and by hypothesis, we have $\dist(z, h_1(z)) < \d$.
	Applying Claim \ref{claim.presmallpurtforclosed ball} to $z$ and $y = h_1(z)$,
	\[
		\widetilde{W}^s(z, \alpha + \e) \cap \widetilde{W}^u(x)= \{[x, z]\} \neq \emptyset
	\]
	and $d^u_x(h_0(h_1(z)), [x, z]) < \e/(2m)$,
	which yields property {\bf (P2')}.
	Thus, we may apply Lemma \ref{lem.generalizedpurt} to obtain that there is a continuous map $h \times \eta : D_1 \to \widetilde{W}^u(x, \beta - \e) \times D^{k+\ell}$ such that $(D_1, h \times \eta) \in \D(\alpha + \e, \widetilde{W}^u(x, \beta - \e) \times D^{k+\ell})$, proving property (a).
	Moreover,
	\begin{align}
		V^u(h \times \eta)
		 & = (h \times \eta)^{-1}({\rm Int} (\widetilde{W}^u(x, \beta - \e) \times D^{k + l}))                                                     \\
		 & \subset ((h_0 \circ h_1) \times ((\eta_0 \circ h_1) \times \eta_1))^{-1}({\rm Int} (\widetilde{W}^u(x, \beta) \times D^{k + l}))        \\
		 & = (h_1 \times \eta_1)^{-1} \circ (h_0 \times \eta_0 \times {\rm id}_{D^{l}})^{-1}({\rm Int} \widetilde{W}^u(x, \beta) \times D^{k + l}) \\
		 & \subset (h_1 \times \eta_1)^{-1}({\rm Int} (D_0 \times D^l)),                                                                           
	\end{align}
	which proves property (b) and finishes the proof of this lemma.
\end{proof}

As in Remark \ref{rk.motivationofD(a(w(xb)))}, $(D, h \times \eta) \in \D(\alpha, \widetilde{W}^u(x, \beta) \times D^k)$ means that $D$ is ``$\alpha$-close'' to $\widetilde{W}^u(x, \beta)$.
The following lemma shows that, if we take $y$ so ``close'' to $x$, then $D$ is ``$(\alpha + \Delta)$-close'' to $\widetilde{W}^u(y, \beta)$, which implies $(D, g \times \theta) \in \D(\alpha + \Delta, \widetilde{W}^u(y, \beta') \times D^k)$. In the proof, we need to make $\beta$ smaller again.
\begin{lem}\label{lem.purtaofbase}
	Let $\alpha, \beta > 0$ be given.
	For every $\beta' \in (0, \beta)$ and every $\Delta > 0$, there exists $d > 0$ such that 
	for all $x, y \in W_i$, $1 \leq i \leq s$, with $\dist(x, y) < d$ and $(D, h \times \eta) \in \D(\alpha, \widetilde{W}^u(x, \beta) \times D^k)$, there is a continuous map $g \times \theta : D \to \widetilde{W}^u(y, \beta') \times D^k$ satisfying the following properties:
	\begin{enumerate}[leftmargin=0.8cm, itemindent=0cm]
		\item[$(a)$]
		      $(D, g \times \theta) \in \D(\alpha + \Delta, \widetilde{W}^u(y, \beta') \times D^k)$;
		\item[$(b)$]
		      $V^u(g \times \theta)
			      \subset V^u(h \times \eta)$.
	\end{enumerate}
\end{lem}
\begin{proof}
	Since $\widetilde{W}^u(z, \beta)$ depends continuously on the base point $z$, there exists $d = d(\beta', \Delta) > 0$ such that for all $x, y \in W_i$ with $\dist(x, y) < d$, there is a homeomorphism $i_y : \widetilde{W}^u(x, \beta) \to \widetilde{W}^u(y, \beta)$ satisfying
	\[
		\dist(z, i_y(z)) < (\beta - \beta')/(4m), \quad z \in \widetilde{W}^u(x, \beta).
	\]
	Moreover, by Claim \ref{claim.condonWi} (2), (making $d$ smaller if necessary), we may assume that if $x \in W_i$ and $w \in U_i$ satisfy $\widetilde{W}^s(w, \alpha) \cap \widetilde{W}^u(x, \beta) = \{[x, w]\} \neq \emptyset$, then the distance between the intersection
	\[
		[y, w] = \widetilde{W}^s(w, \alpha + \Delta) \cap \widetilde{W}^u(y)
	\]
	and $[x, w]$ is less than $(\beta - \beta')/(4m)$ for all $y \in W_i$ with $\dist(x, y) < d$. 
	
	Let us prove that $(i_y \circ h) \times \eta : D \to \widetilde{W}^u(y, \beta) \times D^k$ satisfies properties {\bf (P1)} and {\bf (P2')} of Lemma \ref{lem.generalizedpurt} with $(\alpha, \beta, \e)$ replaced by $(\alpha+\Delta, \beta, \beta - \beta')$.
	As for Property {\bf (P1)}, this follows from Lemma \ref{lem.nontrivial},
	\[
		(i_y \circ h) \times \eta = (i_y \times {\rm id}_{D^k}) \circ (h \times \eta)
	\]
	and the fact that $i_y \times {\rm id}_{D^k}$ is $m_0$-nontrivial.
	As for Property {\bf (P2)}, letting
	\[
		z \in ((i_y \circ h) \times \eta)^{-1}({\rm Int}(\widetilde{W}^u(y, \beta - (\beta - \beta')/2) \times D^k)) \subset (h \times \eta)^{-1}({\rm Int}(\widetilde{W}^u(x, \beta) \times D^k)),
	\]
	we have $h(z) = [x, z]$.
	By the choice of $d$,
	\begin{align}
		\dist(i_y \circ h(z), [y, z]) 
		 & \leq \dist(i_y ([x, z]), [x, z])
		+ \dist([x, z], [y, z])                            \\
		 & < (\beta - \beta')/(4m) + (\beta - \beta')/(4m) \\
		 & = (\beta - \beta')/(2m)
	\end{align}
	for all $z \in ((i_y \circ h) \times \eta)^{-1}({\rm Int}(\widetilde{W}^u(y, \beta - (\beta - \beta')/2) \times D^k))$,
	which implies property {\bf (P2')}.
	Now apply Lemma \ref{lem.generalizedpurt} to obtain a continuous map $g \times \theta : D \to \widetilde{W}^u(y, \beta') \times D^k$ satisfying $(D, g \times \theta) \in \D(\alpha + \Delta, \widetilde{W}^u(y, \beta') \times D^k)$ to have property (b). Finally,
	\begin{align}
		V^u(g \times \theta)
		 & =(g \times \theta)^{-1}({\rm Int}\widetilde{W}^u(y, \beta') \times D^k)                                 \\
		 & \subset ((i_x \circ h) \times \eta)^{-1}({\rm Int}(\widetilde{W}^u(y, \beta) \times D^k))               \\
		 & = (h \times \eta)^{-1}(i_y \times {\rm id}_{D^k})^{-1}({\rm Int}(\widetilde{W}^u(y, \beta) \times D^k)) \\
		 & \subset (h \times \eta)^{-1}({\rm Int}(\widetilde{W}^u(x, \beta) \times D^k)) = V^u(h \times \eta),
	\end{align}
	which proves the property (b) to finish the proof of this lemma.
\end{proof}
\section{The proof of Proposition \ref{pro.starimpliesnocycle}}
In this section, we prove Proposition \ref{pro.starimpliesnocycle}. 
For the proof, we need the following lemma:
\begin{lem}\label{lem.lemforTsuimpliesnocycle}
	Let $f$ be an Axiom A diffeomorphism satisfying the $T^{s, u}$-condition. Given $\alpha, \beta > 0$.
	Suppose that for two periodic points, $p \in \Lambda_i$ and $q \in \Lambda_j$, $1 \leq i, j \leq s$ with $i \neq j$, the stable and unstable manifolds $W^u(p)$ and $W^s(q)$ have an intersection point and $\dim W^u(p) = m - \dim W^s(q) = \dim W^u(q)$.
	Then there exists $k>0$ such that
	\[
		f^{k}(D) \in \D(\alpha, W^u(q, \beta) \times D^0)
	\]
	for all $D \in \D(\alpha, W^u(p, \beta) \times D^0)$.
\end{lem}
\begin{proof}
	Let $u = \dim W^u(p) = \dim W^u(q)$ and let $x \in W^u(p) \cap W^s(q)$.
	Denote the periods of $p$ and $q$ by $\pi_p$ and $\pi_q$, respectively.
	For some $l > 0$, we have $f^{l\pi_p\pi_q}(x) \in W^s(q, \alpha/2)$.
	Let $x_0 = f^{l\pi_p\pi_q}(x)$.
	It follows from Lemma \ref{lem.C0transstarimplyeclosed} that
	there exists a $u$-dimensional embedded closed ball $\hat{D} \subset f^{l\pi_p\pi_q}(W^u(p)) = W^u(p)$ and $\beta' > 0$ such that $\hat{D} \in \D(\alpha, W^u(q, \beta') \times D^0)$.
	Apply Lemma \ref{lem.smallpurtforclosedball} with $\e$ replaced by $\min\{\alpha, \beta'/2\}$ and $k=l=0$ to obtain $\d > 0$ such that if a $u$-dimensional ball $D'$ and a $u$-nontrivial continuous map $h' : D' \to \hat{D}$ satisfy $\dist(y, h'(y)) < \d$ for all $y \in h'^{-1}({\rm Int}\hat{D})$, then $D' \in \D(2\alpha, W^u(q, \beta'/2) \times D^0)$.
	
	
	By Lemma \ref{lem.closeddiscchange} and Lemma \ref{lem.basicofepsilonclosever2}, there exists $h_n \times \eta_n : f^{n \pi_{p}}(D) \to W^u(p, \beta) \times D^0$ such that
	\[
		(f^{n \pi_{p}}(D), h_n \times \eta_n) \in \D(\lambda^{n\pi_{p}}\alpha, W^u(p, \beta) \times D^0)
	\]
	for sufficiently large $n$.
	
	Take $k' > 0$ with $\hat{D} \subset f^{k'}(W^u(p, \beta))$,
	and let $r : f^{k'}(W^u(p, \beta)) \to \hat{D}$ be $u$-nontrivial retraction such that $r(f^{k'}(W^u(p, \beta)) \setminus \hat{D}) = \partial \hat{D}$.
	Then we have the following claim:
	\begin{claim}\label{claim.5}
		For sufficiently large $n$, the map $h'_n = r \circ f^{k'} \circ h_n \circ f^{-k'} : f^{n \pi_{p} + k'}(D) \to \hat{D}$ is $u$-nontrivial and 
		\begin{equation}
			\dist(y, h'_n(y)) < \d, \quad y \in (h'_n)^{-1}({\rm Int}\hat{D}).\label{eq.yh'ny}
		\end{equation}
	\end{claim}
	To prove this claim, first notice that
	\begin{align}
		(h'_n)^{-1}({\rm Int}\hat{D})
		 & = f^{k'} \circ h_n^{-1} \circ f^{-k'} \circ r^{-1}({\rm Int}\hat{D})           \\
		 & = f^{k'} \circ h_n^{-1} \circ f^{-k'}({\rm Int}\hat{D})                        \\
		 & \subset f^{k'} \circ h_n^{-1}({\rm Int}W^u(p, \beta))                          \\
		 & = f^{k'} \circ (h_n \times \eta_n)^{-1}({\rm Int}(W^u(p, \beta) \times D^0)).
	\end{align}
	From this and {\bf (P2)} for $h_n$, it follows that $h_n \circ f^{-k'}(y) = [p, f^{-k'}(y)] \in W^s(f^{-k'}(y), \lambda^{n\pi_{p}}\alpha)$ for all $y \in (h'_n)^{-1}({\rm Int}\hat{D})$.
	Then using \eqref{eq.constofequivmetric}, we see that $\dist (f^{-k'}(y), h_n \circ f^{-k'}(y)) < C_0\lambda^{n\pi_{p}}\alpha$ for all $y \in (h'_n)^{-1}({\rm Int}\hat{D})$.
	This together with the uniform continuity of $f^{k'}$ finishes the proof of this claim for sufficiently large $n$.
	
	By Claim \ref{claim.5}, $h'_n : f^{n \pi_{p} + k'}(D) \to \hat{D}$ is $u$-nontrivial and satisfies \eqref{eq.yh'ny} for sufficiently large $n$.
	From the choice of $\d$, we have $f^{n \pi_{p} + k'}(D) \in \D(2\alpha, W^u(q, \beta'/2) \times D^0)$.
	
	Finally, apply Lemma \ref{lem.closeddiscchange} and Lemma \ref{lem.basicofepsilonclosever2} repeatedly to obtain $n' > 0$ such that
	\[
		f^{n \pi_{p} + k' + n'\pi_q}(D) \in \D(\alpha, W^u(q, \beta) \times D^0).
	\]
	Put $k = n \pi_{p} + k' + n'\pi_q$ to finish the proof of Lemma \ref{lem.lemforTsuimpliesnocycle}.
\end{proof}
{\bf Proof of Proposition \ref{pro.starimpliesnocycle}}.
To get a contradiction, assume that $f$ has a cycle of basic sets.
Then there exist different basic sets $\Lambda_1$, $\Lambda_2$, $\ldots$, $\Lambda_k$
and wandering points $x_1$, $x_2$, $\ldots$, $x_k$ such that 
\[
	x_i \in W^u(\Lambda_i) \cap W^s(\Lambda_{i+1}), \quad i = 1, 2, \ldots, k,
\]
where $\Lambda_{k+1} = \Lambda_1$.
Since $x_{i+1} \in W^u(\Lambda_{i+1})$ and $x_i \in W^s(\Lambda_{i+1})$,
we have $\dim W^s(x_{i}) + \dim W^u(x_{i+1}) = m$ for all $ i = 1, 2, \ldots, k$.
Combining this and Remark \ref{rk.homology} (1), we obtain
\[
	km
	= \sum^k_{i=1} (\dim W^s(x_{i}) + \dim W^u(x_{i+1}))
	= \sum^k_{i=1} (\dim W^s(x_i) + \dim W^u(x_i)) \geq km.
\]
Thus, $\dim W^s(x_i) + \dim W^u(x_i) = m$ for all $ i = 1, 2, \ldots, k$.
This implies that $\dim W^u(x_i)$ is constant for all $ i = 1, 2, \ldots, k$. Let $u$ be the dimension of the unstable manifold $W^u(x_1)$.

Now take $r_i \in \Lambda_i$ and $p_{i+1} \in \Lambda_{i+1}$ with $x_i \in W^u(r_i) \cap W^s(p_{i+1})$.
By Remark \ref{rk.PETROV2015prop2}, $x_i$ is a $C^0$-transversal intersection of $W^u(r_i)$ and $W^s(p_{i+1})$ for all $i = 1, 2, \ldots, k$ (where $p_{k+1} = p_1$). Thus, there is $\overline{d} > 0$ such that there is $x'_i \in W^u(r'_i) \cap W^s(p'_{i+1})$ near $x_i$ for all $r'_i \in \Lambda_i$ and $p'_{i+1} \in \Lambda_{i+1}$ with $\max \{ \dist(r'_i, r_i), \dist(p'_{i+1}, p_{i+1}) \} < \overline{d}$.
Since $f$ satisfies the $T^{s, u}$-condition, we see that  the $T^{s, u}$-condition is satisfies at $x'_i$.
It follows from the fact that periodic points are dense in each basic set $\Lambda_i$ that there is a periodic point $p'_i \in \Lambda_i$ and $k_i > 0$ such that
\[
	\dist(p'_i, p_i), \dist(f^{k_i}(p'_i), r_i) < \overline{d}, \quad i = 1, 2, \ldots, k.
\]
Replacing $(p'_i, f^{k_i}(p'_i), x'_i)$ by $(p_i, r_i, x_i)$, we have that:
\begin{itemize}
	\item $p_i$ and $r_i$ are in the same periodic orbit contained in $\Lambda_i$, and $r_i = f^{k_i}(p_i)$ for each $i = 1, 2, \ldots, k$.
	\item $x_i$ is a wandering point in $W^u(r_i) \cap W^s(p_{i+1})$ for each $i = 1, 2, \ldots, k$.
\end{itemize}
Since $x_1$ is a wandering point, there is a neighborhood $U$ of $x_1$ satisfying $f^n(U) \cap U = \emptyset$ for all $n \geq 1$,
containing an embedded $u$-dimensional closed ball $D$ transversal to $W^s(p_2)$. 
Since $p_2$ is periodic, $f^\ell(D)$ and $W^s(p_2, \alpha/2)$ have a transversal intersection for some $\ell>0$.
In particular, the intersection of $f^\ell(D)$ and $W^s(p_2, \alpha/2)$ satisfies the $T^{s,u}$-condition.

By Lemma \ref{lem.C0transstarimplyeclosed}, we have $\beta > 0$ and $D_0 \subset f^l(D)$ with $D_0 \in \D(\alpha, W^u(p_2, \beta) \times D^0)$.
Then, apply Lemma \ref{lem.closeddiscchange} and \ref{lem.basicofepsilonclosever2} repeatedly to see that 
\[
	f^{k_2}(D_0) \in \D(\lambda^{k_2}\alpha, W^u(r_2, \beta) \times D^0),
\]
and then Lemma \ref{lem.lemforTsuimpliesnocycle} with $(p, q)$ replaced by $(r_2, p_3)$ to show
\[
	f^{k_2+\ell_2}(D_0) \in \D(\alpha, W^u(p_3, \beta) \times D^0)
\]
for some $l_2 > 0$.
Now continue this process for $r_3$, $p_4$, and so on.
As a consequence, we deduce that there is $L > 0$ such that
\[
	f^{L}(D_0) \in \D(\alpha, W^u(r_1, \beta) \times D^0).
\]
Let $\pi_1$ be the period of $r_1$.
Then, apply Lemma \ref{lem.closeddiscchange} and \ref{lem.basicofepsilonclosever2} repeatedly to see 
\[
	f^{L+n\pi_1}(D_0) \in \D(\lambda^{n\pi_1}\alpha, W^u(r_1, \beta) \times D^0), \quad n \geq 1.
\]
Since $f^{L+n\pi_1}(D_0)$ is ``$\lambda^{n\pi_1}\alpha$-close'' to $W^u(r_1, \beta)$ in the sense of Remark \ref{rk.motivationofD(a(w(xb)))},
we have $f^{L+n\pi_1}(D_0) \cap \Orb^-(U) \neq \emptyset$.
This and $f^{L+n\pi_1}(D_0) \subset f^{L+n\pi_1+\ell}(D) \subset \Orb^+(U)$ exhibit a contradiction.

\qed
\section{Proof of Theorem \ref{thm.homologyshadowing}}
In this section, we prove Theorem \ref{thm.homologyshadowing} by extending the proof in \cite{P.S.trans} through homological methods in Section \ref{sec.homologicalinclinationlemma}.
Suppose that $f$ is an Axiom A diffeomorphism satisfying the $T^{s, u}$-condition.

For a pseudotrajectory $\xi = \{ x_k ; k \in \Z \}$ and integers $l \leq l'$, we denote
\[
	\xi^{l, l'} = \{ x_k ; l \leq k \leq l' \}, \quad \xi^l_- = \{ x_k ; k \leq l \}, \quad \text{and} \quad \xi^l_+ = \{ x_k ; k \geq l \}.
\]
For two basic sets $\Lambda_i$ and $\Lambda_j$, we write $\Lambda_i \to \Lambda_j$ if $W^u(\Lambda_i) \cap W^s(\Lambda_j) \setminus (\Lambda_i \cup \Lambda_j) \neq \emptyset$.

To prove the main theorem, we classify the pseudotrajectories by their staying time in neighborhoods of hyperbolic sets and the time spending between the neighborhoods of hyperbolic sets (Figure \ref{figure.PT(LKs)}).
\begin{dfn}\label{dfn.PT(LKs)}
	For $L, N, s_0 > 0$ and $d>0$,
	denote by $PT(L, N, s_0, d)$ the set of all $d$-pseudotrajectories $\xi = \{x_k\}$ of $f$ such that there exist basic sets $\Lambda_1, \ldots, \Lambda_{\overline{s}}$ with $0 \leq \overline{s} \leq s_0$, indices $l \in \Z$ and
	\begin{equation}
		0 = \tau(1) < t(2) \leq \tau(2) < t(3) \leq \tau(3) < \cdots < t(\overline{s}-1) \leq \tau(\overline{s}-1) < t(\overline{s})\label{eq.taut}
	\end{equation}
	satisfying:
	\begin{itemize}[leftmargin=0.4cm, itemindent=0cm]
		\item  $0 < t(i+1) - \tau(i) \leq L$, $i = 1, \ldots, \overline{s}-1$;
		\item $\tau(i) - t(i) \geq N$, $i = 2, \ldots, \overline{s}-1$;
		\item $\xi^{l+t(i), l + \tau(i)} \in W_i$, $i = 2, \ldots, \overline{s}-1$; and
		\item $\xi^{l+\tau(1)}_- \in W_1$ and $\xi^{l+t(\overline{s})}_+ \in W_{\overline{s}}$.
	\end{itemize}
\end{dfn}
\begin{figure}[h]
	\centering
	\includegraphics[width=12cm]{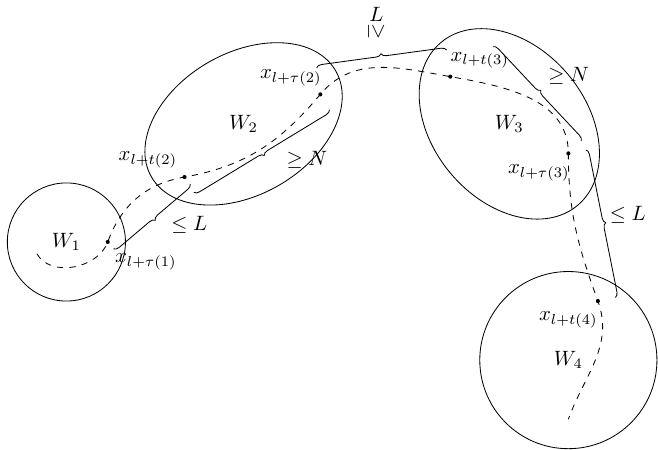}
	\caption{$\xi = \{x_k\} \in PT(L, N, s_0, d)$. In this case, $\overline{s} = 4$.}\label{figure.PT(LKs)}
\end{figure}

We need the following two lemmas in \cite{TransandCopensets}.
\begin{lem}[{\cite[Lemma 4.1]{TransandCopensets}}]\label{lem.birkhoffconstant}
	Let $f$ be a homeomorphism of a compact metric space and let $U$ be a neighborhood of $\Omega(f)$. Then, there exist positive numbers $d, L > 0$ such that if $\xi = \{ x_k ; k \in \Z \}$ is a $d$-pseudotrajectory of $f$ and $\xi^{l, l'} \cap U = \emptyset$ for some $l \leq l'$, then $l' - l \leq L$.
\end{lem}
By Proposition \ref{pro.starimpliesnocycle} and our hypothesis, note that $f$ is $C^1$ $\Omega$-stable.
\begin{lem}[{\cite[Lemma 4.2]{TransandCopensets}}]\label{lem.connofbasicsets}
	Let $f$ be a $C^1$ $\Omega$-stable diffeomorphism. Then there exist compact neighborhoods $V_i$, $i=1, 2, \ldots, s$, of basic sets of $f$ and a number $d > 0$ with the following property:
	if $\xi = \{x_k\}$ is a $d$-pseudotrajectory of $f$ such that $x_0 \in V_i$ and $x_l \in V_j$ for some $l > 0$, then there exist basic sets $\Lambda_{l_i}, \ldots, \Lambda_{l_k}$ such that
	\[
		\Lambda_i \to \Lambda_{l_1} \to \cdots \to \Lambda_{l_k} \to \Lambda_j.
	\]
\end{lem}

First we apply Lemmas \ref{lem.birkhoffconstant} and \ref{lem.connofbasicsets} and replace the neighborhoods $\{W_i\}$ smaller if necessary, to find numbers $d_0, L_0 > 0$ such that $\xi \in PT(L_0,1,s,d_0)$ for all $d_0$-pseudotrajectories $\xi$.
In addition, by \cite{S.P.}, we may also assume that $f$ has the Lipschitz shadowing property in each $W_i$.
\begin{rk}
	It is obvious that the lemmas in Section \ref{sec.homologicalinclinationlemma} still hold even if we replace $W_i$, $1 \leq i \leq s$, by some smaller ones.
\end{rk}

Theorem \ref{thm.homologyshadowing} follows from the following two propositions:
\begin{pro}\label{pro.eachstep}
	For every $\e > 0$ and $L \geq L_0$, there exist $d>0$ and $N \in \Z_{>0}$ such that if $\xi \in PT(L,N+1,s,d)$, then $\xi$ is $\e$-shadowed by some point $x \in M$.
\end{pro}
\begin{pro}\label{pro.induction}
	Let $L \geq L_0$, $N \in \Z_{> 0}$, $3 \leq s_0 \leq s$ and $d > 0$ be given.
	Then
	\[
		PT(L, 1, s_0, d) \subset PT(L,N+1,s_0, d) \cup PT((s_0-1)L+(s_0-2)N, 1, s_0-1,d).
	\]
\end{pro}
\begin{proof}
	Let $\xi \in PT(L, 1, s_0, d) \setminus PT(L,N+1,s_0, d)$.
	Since $\xi \in PT(L, 1, s_0, d)$, there exist basic sets $\Lambda_1, \ldots, \Lambda_{\overline{s}}$ with $0 \leq \overline{s} \leq s_0$ and
	\[
		0 = \tau(1) < t(2) \leq \tau(2) < t(3) \leq \tau(3) < \cdots < t(\overline{s}-1) \leq \tau(\overline{s}-1) < t(\overline{s})
	\]
	satisfying the four condition on Definition \ref{dfn.PT(LKs)} for some $l \in \Z$.
	Let
	\begin{align}
		{\cal I} & = \{ 1, \overline{s} \} \cup \{ 1 < i < \overline{s} ; \tau(i) - t(i) \geq N + 1 \} \\
		         & = \{1 = i_1 < i_2 < \cdots < i_{k-1} < i_k = \overline{s}\}.
	\end{align}
	Notice that $k < \overline{s} \leq s_0$ since $\xi \notin PT(L,N+1,s_0,d)$.
	Let $L' = (s_0-1)L+(s_0-2)N$.
	We need to prove that basic sets $\Lambda_{i_1}, \ldots, \Lambda_{i_k}$ with $0 \leq k \leq s_0-1$ and
	\[
		0 = \tau({i_1}) < t({i_2}) \leq \tau({i_2}) < t({i_3}) \leq \tau(i_3) < \cdots < t(i_{k-1}) \leq \tau(i_{k-1}) < t(i_k)
	\]
	satisfy the four condition on Definition \ref{dfn.PT(LKs)} for some $l \in \Z$ with $(L, N, s_0,d)$ replaced by $(L', N+1, s_0-1,d)$. This implies $\xi \in PT((s_0-1)L+(s_0-2)N, 1, s_0-1, d)$ as required.
	In fact, for each $j = 1, \ldots, k-1$, we have
	\begin{align}
		t(i_{j+1}) - \tau(i_j)
		 & = (t(i_{j+1}) - \tau(i_{j+1}-1))
		+ (\tau(i_{j+1}-1) - t(i_{j+1}-1))     \\
		 & + (t(i_{j+1}-1) - \tau(i_{j+1}-2))
		+ (\tau(i_{j+1}-2) - t(i_{j+1}-2))     \\
		 & + \cdots + (t(i_j+1) - \tau(i_j))   \\
		 & \leq L + N + L + N + \cdots + L     \\
		 & = (i_{j+1}-i_j)L + (i_{j+1}-i_j-1)N \\
		 & \leq (s_0-1)L + (s_0-2)N            \\
		 & =L',
	\end{align}
	proving the first item.
	Other three items are trivial.
\end{proof}
Before proving Proposition \ref{pro.eachstep}, let us give the proof of Theorem \ref{thm.homologyshadowing}.
\vspace{3mm}
\\
{\bf Proof of Theorem \ref{thm.homologyshadowing}}.
Fix an arbitrary $\e > 0$.
By the choice of $L_0$ and $d_0$, we see that
\begin{equation}
	\xi \in PT(L_0,1,s,d)\label{eq.setofpt}
\end{equation}
for all $d \in (0, d_0)$ if $\xi$ is a $d$-pseudotrajectory.
Applying Proposition \ref{pro.eachstep} with $L = L_0$, we have $d_1 \in (0, d_0)$ and $N_0 \in \Z_{>0}$ such that every $\xi \in PT(L_0,N_0+1,s,d_1)$ is $\e$-shadowed by some point $x \in M$.
Let $L_1 = (s-1)L_0 + (s-2)N_0$.
Again, applying Proposition \ref{pro.eachstep} with $L = L_1$, we have $d_2 \in (0, d_1)$ and $N_1 \in \Z_{>0}$ such that every $\xi = \{x_k\} \in PT(L_1,N_1+1,s-1,d_2)$ is $\e$-shadowed by some point $x \in M$.

Repeating this procedure, we can define $L_0 < L_1 < \cdots < L_{s-2}$, $N_0 < N_1 < \cdots < N_{s-3}$ and $d_1 > d_2 > \cdots > d_{s-2}$ to satisfy the following properties:
\begin{itemize}[leftmargin=0.8cm, itemindent=0cm]
	\item[(a)] Every $\xi \in PT(L_i,N_i+1,s-i,d_{i+1})$ is $\e$-shadowed by some point $x \in M$ for all $i = 0, 1, \ldots, s-3$;
	\item[(b)] $PT(L_i, 1, s-i, d) \subset PT(L_i, N_i+1, s-i, d) \cup PT(L_{i+1}, 1, s-(i+1), d)$ for all $i = 0, 1, \ldots, s-3$ and $d > 0$ (this follows from Proposition \ref{pro.induction}).
\end{itemize}
Applying Proposition \ref{pro.eachstep} again with $L = L_{s-2}$, we have $d \in (0, d_{s-2})$ and $N \in \Z_{>0}$ such that every $\xi \in PT(L_{s-2},N,2,d)$ is $\e$-shadowed by some point $x \in M$.
By Definition \ref{dfn.PT(LKs)},
\[
	PT(L_{s-2},1,2, d) = PT(L_{s-2},N, 2 ,d)
\]
(since these sets does not depend on $N$ when $s_0 \leq 2$).

Now use property (b) repeatedly to obtain that
\begin{align}
	        & PT(L_0,1,s,d)                                                                                            \\
	\subset & PT(L_0, N_0+1, s, d) \cup PT(L_1, 1, s-1, d)                                                             \\
	\subset & PT(L_0, N_0+1, s, d) \cup (PT(L_1, N_1+1, s-1, d) \cup PT(L_2, 1, s-2, d))                               \\
	\subset & PT(L_0, N_0+1, s, d) \cup (PT(L_1, N_1+1, s-1, d) \cup (PT(L_2, N_2+1, s-2, d) \cup PT(L_3, 1, s-3, d))) \\
	\subset & \cdots                                                                                             
	\subset \left( \bigcup_{i=0}^{s-3} PT(L_i,N_i+1,s-i,d_{i+1}) \right) \cup PT(L_{s-2},1,2, d)                       \\
	=       & \left( \bigcup_{i=0}^{s-3} PT(L_i,N_i+1,s-i,d_{i+1}) \right) \cup PT(L_{s-2},N,2, d).
\end{align}
Thus, by the choice of $N$ and property (a), $\xi$ is $\e$-shadowed by some point $x \in M$ if $\xi \in PT(L_0,1,s,d)$.
Thus, it follows from \eqref{eq.setofpt} that every $d$-pseudotrajectory $\xi$ belongs to $PT(L_0,1,s,d)$.
This completes the proof of Theorem \ref{thm.homologyshadowing}.

In order to prove Proposition \ref{pro.eachstep}, we need two propositions.

The first one is a multi-dimensional analog of \cite[Lemma 4.3.]{P.S.trans}. Fix $\mu \in (\lambda, 1)$.
\begin{pro}\label{pro.generalizedcloseddiscchange}
	Let $\alpha, \beta > 0$ be given. Suppose that $\Delta > 0$ satisfies the inequality
	\begin{equation}
		\beta/\lambda - \Delta > \beta/\mu.\label{eq.generalizedcloseddiscchange}
	\end{equation}
	Then there exists a number $d > 0$ depending on $\alpha$, $\beta$, and $\Delta$ such that:
	\begin{itemize}[leftmargin=0.6cm, itemindent=0cm]
		\item 
	If $x, f(x), y \in W_i$, $\dist(f(x), y) < d$, and $(D, h \times \eta) \in \D(\gamma, \widetilde{W}^u(x, \beta) \times D^k)$ for some $k \geq 0$ and $\gamma \in (0, \alpha)$, then there is a continuous map $g \times \theta : f(D) \to \widetilde{W}^u(y, \beta/\mu) \times D^k$ for which
	\[
		(f(D), g \times \theta) \in \D(\lambda\gamma + \Delta, \widetilde{W}^u(y, \beta/\mu) \times D^k)
	\]
	and
	\[
		V^u(g \times \theta) \subset f (V^u(h \times \eta))).
	\]
\end{itemize}
\end{pro}
\begin{proof}
	Let $\Delta > 0$ be such that 
	\[
		\beta/\lambda - \Delta > \beta/\mu.
	\]
	Apply Lemma \ref{lem.closeddiscchange} with $(\alpha, m_0-u)$ replaced by $(\gamma, k)$ to have a continuous map $g' \times \theta' : f(D) \to \widetilde{W}^u(f(x), \beta/\lambda) \times D^k$ such that
	$(f(D), g' \times \theta') \in \D(\lambda\gamma, \widetilde{W}^u(f(x), \beta/\lambda) \times D^k)$ and
	\begin{equation}
		V^u(g' \times \theta')
		\subset f(V^u(h \times \eta)).\label{eq.incofcloseddiscchange}
	\end{equation}
	Next apply Lemma \ref{lem.purtaofbase} with $(\alpha, \beta)$ replaced by $(\lambda\gamma, \beta/\lambda)$ to have a continuous map $g \times \theta : f(D) \to \widetilde{W}^u(y, \beta/\mu) \times D^k$ such that $(f(D), g \times \theta) \in \D(\lambda\gamma + \Delta, \widetilde{W}^u(y, \beta/\mu) \times D^k)$ and
	\[
		V^u(g \times \theta) \subset V^u(g' \times \theta').
	\]
	Here we need $\dist(f(x),y) < d$ with sufficiently small $d > 0$ in order to apply Lemma \ref{lem.purtaofbase}.
	From this and \eqref{eq.incofcloseddiscchange}, we finish the proof.
\end{proof}
The second one is the following proposition. 
\begin{pro}\label{pro.betweenbasicsets}
	For $\alpha, \beta, L > 0$, there exist numbers $\Delta = \Delta(\alpha, \beta, L)$,
	$d = d(\alpha, \beta, L)$, $\beta' = \beta'(\alpha, \beta, L) \leq \beta$, and $N = N(\alpha, \beta, L)$ satisfying the following property:
	\begin{itemize}[leftmargin=0.6cm, itemindent=0cm]
		\item
		      If $\xi = \{ x_\nu \}$ is a $d$-pseudotrajectory of $f$ such that $\xi^{-N, 0} \subset W_i$, $\xi^{\ell, \ell + N} \subset W_j$ for some $i \neq j$ with $1 \leq i, j \leq s$ and $\ell \in [0, L]$, and
		      \begin{equation}
			      (D, g \times \theta) \in \D(\Delta, \widetilde{W}^u(x_0, \beta) \times D^k)\label{eq.betweenbasicsets}
		      \end{equation}
		      for some $0 \leq k \leq m - \dim\widetilde{W}^u(x_0)$,
		      then there exist $k' \geq 0$ with $\dim \widetilde{W}^u(x_0) + k = \dim \widetilde{W}^u(x_\ell) + k'$ and a continuous map $h \times \eta: f^\ell(D) \to \widetilde{W}^u(x_\ell, \beta') \times D^{k'}$ such that 
		      \[
			      (f^\ell(D), h \times \eta) \in \D(\alpha, \widetilde{W}^u(x_\ell, \beta') \times D^{k'})
		      \]
		      and
		      \[
			      V^u(h \times \eta) \subset f^\ell(V^u(g \times \theta)).
		      \]
	\end{itemize}
\end{pro}
In order to prove Proposition \ref{pro.betweenbasicsets}, we need a ``local version'' of the proposition.
\begin{lem}\label{lem.betweenbasicsets}
	Let $\alpha, \beta > 0$ be given.
	Assume that $p \in W_i$ and $q \in W_j$, $1 \leq i, j \leq s$ with $i \neq j$, satisfy $f^\tau(p) = q$ for some $\tau \geq 1$ and
	\[
		\Orb^-(p) \subset W_i, \quad \Orb^+(q) \subset W_j.
	\]
	Then there exist numbers $\Delta$,
	$\beta' \leq \beta$, and $\Delta' > 0$ satisfying the following property:
	\begin{itemize}[leftmargin=0.6cm, itemindent=0cm]
		\item If $p' \in W_i$ and $q' \in W_j$ satisfy $\max \{ \dist(p', p), \dist(q', q) \} < \Delta'$, and
		      \begin{equation}
			      (D, g \times \theta) \in \D(\Delta, \widetilde{W}^u(p', \beta) \times D^k)\label{eq.lem_betweenbasicsets}
		      \end{equation}
		      for some $0 \leq k \leq m - \dim\widetilde{W}^u(p')$,
		      then there exist $k' \geq 0$ with $\dim \widetilde{W}^u(p') + k = \dim \widetilde{W}^u(q') + k'$ and a continuous map $h \times \eta: f^\tau(D) \to \widetilde{W}^u(q', \beta') \times D^{k'}$ such that 
		      \[
			      (f^\tau(D), h \times \eta) \in \D(\alpha, \widetilde{W}^u(q', \beta') \times D^{k'})
		      \]
		      and
		      \[
			      V^u(h \times \eta) \subset f^\ell(V^u(g \times \theta)).
		      \]
	\end{itemize}
\end{lem}
\begin{proof}
	Let $u_p = \dim \widetilde{W}^u(p) = \dim \widetilde{W}^u(p')$ and let $u_q = \dim \widetilde{W}^u(q) = \dim \widetilde{W}^u(q)$.
	Since $\Orb^-(p) \subset W_i$ and $\Orb^+(q) \subset W_j$, there exist points $P \in \Lambda_I$ and $Q \in \Lambda_J$ such that $p \in W^u(P)$ and $q \in W^s(Q)$.
	Then $\widetilde{W}^u(p, \beta) \subset W^u(P)$ and $\widetilde{W}^s(q, \alpha) \subset W^s(Q)$.
	Since $f$ satisfies the $T^{s,u}$-condition, we see that $\widetilde{W}^s(q, \alpha/8)$ and $f^\tau(\widetilde{W}^u(p, \beta))$ satisfy the $T$-condition at $q$.
	By Remark \ref{rk.homology} (1), $\dim \widetilde{W}^u(p) \geq m - \dim \widetilde{W}^s(q) = \dim \widetilde{W}^u(q)$.
	Note that $\dim \widetilde{W}^u(p) \geq 1$ by Lemma \ref{lem.connofbasicsets} (for otherwise $\dim \widetilde{W}^u(p) = 0$, and then $W^u(\Lambda_i) = \Lambda_i$ and hence $\Lambda_i \to \Lambda_j$ does not hold).
	
	From Lemma \ref{lem.C0transstarimplyeclosed} with $(x, d, u, k)$ replaced by $(q, u_p, u_q, u_p - u_q)$, it follows that there exists $\hat{\beta} \in (0, \beta)$ and an embedded $u_p$-dimensional closed ball $\hat{D} \subset f^\tau(\widetilde{W}^u(p, \beta))$ such that $(\hat{D}, \hat{h}_0 \times \hat{\eta}_0) \in \D(\alpha/4, \widetilde{W}^u(q, \hat{\beta}) \times D^{u_p-u_q})$ for some continuous map $\hat{h}_0 \times \hat{\eta}_0 : \hat{D} \to \widetilde{W}^u(q, \hat{\beta}) \times D^{u_p-u_q}$.
	Let $\d > 0$ be the constant given by Claim \ref{claim.presmallpurtforclosed ball} with $(\alpha, \beta, \e, D_0)$ replaced by $(\alpha/4, \hat{\beta}, \min(\alpha/4, \hat{\beta}/2), \hat{D})$.
	Let $r : f^\tau(\widetilde{W}^u(p, \beta)) \to \hat{D}$ be a $u_p$-nontrivial retraction such that $r(f^\tau(\widetilde{W}^u(p, \beta)) \setminus \hat{D}) = \partial \hat{D}$.
	Then the uniform continuity of $f^\tau$
	implies that there is $\d' > 0$ such that $\dist(f^\tau(z), f^\tau(w)) < \d$ for all $z, w \in M$ with $\dist(z, w) < \d'$.
	By continuity of $\{ \widetilde{W}^u(x) ; x \in W_i \}$, we can find $\Delta' > 0$ such that if $p' \in W_i$ with $\dist(p', p) < \Delta'$, some homeomorphism $i_{p'} : \widetilde{W}^u(p', \beta) \to \widetilde{W}^u(p, \beta)$ can be chosen to satisfy
	\begin{equation}
		\sup \{ \dist(z, i_{p'}(z)) ;  z \in \widetilde{W}^u(p', \beta) \}  < \d'/2.\label{eq.defofin}
	\end{equation}
	Then it follows from {\bf (P2)} that there is $\Delta > 0$ such that every $(D, g \times \theta) \in \D(\Delta, \widetilde{W}^u(p', \beta) \times D^k)$ satisfies
	\begin{equation}
		\dist(y, g(y)) < \d'/2\label{eq.choiceofDelta}
	\end{equation}
	for all $y \in V^u(g \times \theta)$.
	
	Let us prove that $\Delta, \hat{\beta}/4$ and $\Delta'$ satisfy the property in the statement of this lemma.
	Let $(D, g \times \theta) \in \D(\Delta, \widetilde{W}^u(p', \beta) \times D^k)$. Then let us prove the following claim:
	\begin{claim}\label{claim.6}
		Let $D' = f^\tau(D)$ and let $G : D' \to \hat{D}$ be a continuous map defined by
		\[
			G(y) = r \circ f^\tau \circ i_{p'} \circ g \circ f^{-\tau}(y)
		\]
		(see Figure \ref{figure.Gn}).
		Then the following two properties hold:
		\begin{itemize}[leftmargin=0.4cm, itemindent=0cm]
			\item $G \times (\theta \circ f^{-\tau}) : D' \to \hat{D} \times D^k$ is $(u_p+k)$-nontrivial;
			\item $\dist(x, G(x)) < \d$ for all $x \in (G \times (\theta \circ f^{-\tau}))^{-1}({\rm Int}(\hat{D} \times D^k))$.
		\end{itemize}
	\end{claim}
	\begin{figure}[h]
		\centering
		\includegraphics[width=15cm]{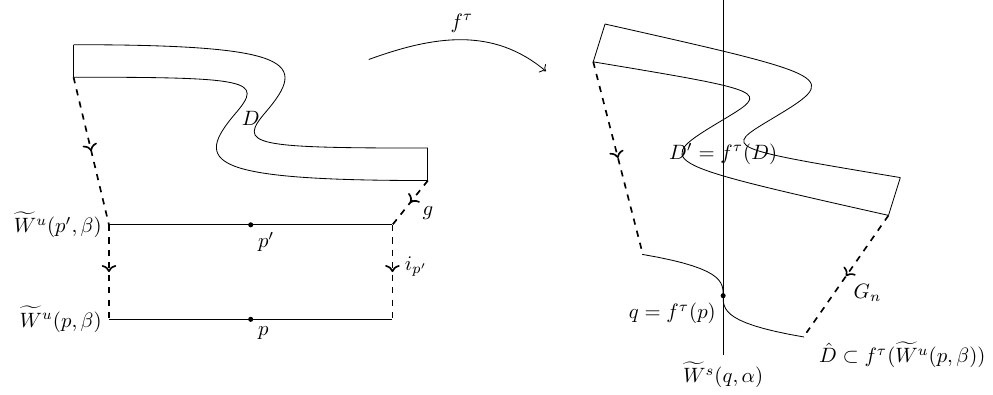}
		\caption{The construction of $G$}\label{figure.Gn}
	\end{figure}
	The map $G \times (\theta \circ f^{-\tau})$ can be decomposed into the following several maps;
	\[
		D' \xrightarrow{f^{-\tau}}
		D \xrightarrow{g \times \theta}
		\widetilde{W}^u(p', \beta) \times D^{k} 
		\xrightarrow{(f^\tau \circ i_{p'}) \times {\rm id}_{D^k}}
		f^\tau(\widetilde{W}^u(p, \beta)) \times D^{k}
		\xrightarrow{r \times {\rm id}_{D^k}}
		\hat{D} \times D^{k}.
	\]
	From this and Lemma \ref{lem.nontrivial}, it follows that the map $G \times (\theta \circ f^{-\tau})$ is $(u_p+k)$-nontrivial.
	
	Consider
	\begin{align}
		        & (G \times (\theta \circ f^{-\tau}))^{-1}({\rm Int}(\hat{D} \times D^k))                                                               \\
		=       & f^{\tau} \circ (g \times \theta)^{-1} \circ ((r \circ f^\tau \circ i_{p'}) \times {\rm id}_{D^k})^{-1}({\rm Int}(\hat{D} \times D^k)) \\
		\subset & f^{\tau} \circ (g \times \theta)^{-1} ({\rm Int}(\widetilde{W}^u(p',\beta) \times D^k))                                               \\
		=       & f^{\tau}(V(g \times \theta)).\label{eq.VG_n_times_theta_n_ftau}
	\end{align}
	Using \eqref{eq.choiceofDelta}, we see that
	\[
		\dist(f^{-\tau}(y), g \circ f^{-\tau}(y)) < \d'/2
	\]
	for all $y \in (G \times (\theta \circ f^{-\tau}))^{-1}({\rm Int}(\hat{D} \times D^k))$.
	By this and \eqref{eq.defofin},
	\begin{align}
		\dist(f^{-\tau}(y), i_{p'} \circ g \circ f^{-\tau}(y))
		 & \leq \dist(f^{-\tau}(y), g \circ f^{-\tau}(y))                   \\
		 & + \dist(g \circ f^{-\tau}(y), i_{p'} \circ g \circ f^{-\tau}(y)) \\
		 & < \d'/2 + \d'/2                                                  \\
		 & = \d'.
	\end{align}
	From the choice of $\d'$, we have
	\[
		\sup \{\dist(x, G(x)) ; x \in (G \times (\theta \circ f^{-\tau}))^{-1}({\rm Int}(\hat{D} \times D^k))\} < \d,
	\]
	which proves Claim \ref{claim.6}.
	
	Apply Lemma \ref{lem.smallpurtforclosedball} with $(\alpha, \beta, \e, u, k, l, D_0, D_1)$ replaced by
	\[
		(\alpha/4, \hat{\beta}/2, \min \{ \alpha/2, \hat{\beta}/2\}, u_q, u_p-u_q, k, \hat{D}, D')
	\]
	to have a continuous map $h' \times \eta' : D' \to \widetilde{W}^u(q, \hat{\beta}/2) \times D^{u_p-u_q + k}$ satisfying
	$(D', h' \times \eta') \in \D(\alpha/2, \widetilde{W}^u(q, \hat{\beta}/2) \times D^{u_p-u_q + k})$ and
	\begin{equation}
		V^u(h' \times \eta')
		\subset (G \times (\theta \circ f^{-\tau}))^{-1}({\rm Int} (\hat{D} \times   D^k)).
		\label{eq.inclusionbetweenbasicsetssono1}
	\end{equation}
	Then apply Lemma \ref{lem.purtaofbase} with $(\alpha, \beta, \beta', \Delta)$ replaced by $(\alpha/2, \hat{\beta}/2, \hat{\beta}/4, \alpha/2)$ to have a continuous map $h \times \eta : D' \to \widetilde{W}^u(q', \hat{\beta}/4) \times D^{u_p-u_q + k}$ satisfying $(D', h \times \eta) \in \D(\alpha, \widetilde{W}^u(q', \hat{\beta}/4) \times D^{u_p-u_q + k})$ and
	\begin{equation}
		V^u(h \times \eta)
		\subset V^u(h' \times \eta').\label{eq.inclusionbetweenbasicsetssono2}
	\end{equation}
	Now Use \eqref{eq.VG_n_times_theta_n_ftau}, \eqref{eq.inclusionbetweenbasicsetssono1} and \eqref{eq.inclusionbetweenbasicsetssono2} to obtain
	\[
		V^u(h \times \eta) \subset f^{\tau}(V(g \times \theta)).
	\]
	
	Thus, we have proved that $\Delta, \beta' = \hat{\beta}/4$ and $\Delta'$ satisfy the property in the statement of this lemma,
	finishing the proof.
\end{proof}
For the proof of Proposition \ref{pro.betweenbasicsets}, we assume the contrary. In this case, we can find $\tau_n \in \Z$ with $0 \leq \tau_n \leq L$, $\Delta_n, d_n, \beta_n \to 0$ and $N(n) \to \infty$ as $n \to \infty$ and $d_n$-pseudotrajectory $\xi^{(n)} = \{ x^{(n)}_k \}_k$ satisfying the following properties:
\begin{enumerate}[leftmargin=0.8cm, itemindent=0cm]
	\item[(a)] $x^{(n)}_{\ell} \in W_{i(n)}$ for $-N(n) \leq \ell \leq 0$ and $x^{(n)}_{\tau_n + \ell} \in W_{j(n)}$ for $0 \leq \ell \leq N(n)$;
	\item[(b)] For every $n \geq 1$, there exists $(D_n, g_n \times \theta_n) \in \D(\Delta_n, \widetilde{W}^u(x^{(n)}_0, \beta) \times D^{k_n})$ with $k_n \geq 0$ such that the following property {\bf does not} hold:
	      \begin{quote}
		      There exist $k'_n \geq 0$ and a continuous map $h_n \times \eta_n : f^{\tau_n}(D_n) \to \widetilde{W}^u(x^{(n)}_{\tau_n}, \beta_n) \times D^{k'_n}$ with $\dim \widetilde{W}^u(x^{(n)}_0) + k_n = \dim \widetilde{W}^u(x^{(n)}_\ell) + k'_n$ such that $(f^{\tau_n}(D_n), h_n \times \eta_n) \in \D(\alpha, \widetilde{W}^u(x^{(n)}_{\tau_n}, \beta_n) \times D^{k'_n})$ and
		      \[
			      V^u(h_n \times \eta_n) \subset f^{\tau_n}(V^u(g_n \times \theta_n)).
		      \]
	      \end{quote}
\end{enumerate}
Passing to a subsequence, we may assume that:
\begin{itemize}[leftmargin=0.4cm, itemindent=0cm]
	\item $x^{(n)}_{0} \to p \in W_I$, $x^{(n)}_{\tau_n} \to q \in W_J$ as $n \to \infty$ for some $W_I, W_J$ (since $\{W_i\}$ are compact); and
	\item Using the pigeonhole principle, $\tau_n = \tau \in [0, L]$, $\dim\widetilde{W}^u(x^{(n)}_0) = u$, $\dim\widetilde{W}^u(x^{(n)}_\tau) = u'$ and $k_n = k$ for some $\tau, u, u', k$ (notice that $f^\tau(p) = q$).
\end{itemize}
Since $N(n) \to \infty$ and $d_n \to 0$ as $n \to \infty$, we have
\[
	f^a(p) = \lim_{n \to \infty} f^a(x^{(n)}_{0}) = \lim_{n \to \infty} x^{(n)}_{a} \in W_I
\]
for all $a \leq 0$.
Similarly, we have $f^a(q) \in W_J$ for all $a \geq 0$.
Let $\Delta$,
$\beta' \leq \beta$, and $\Delta' > 0$ be the constants given by Lemma \ref{lem.betweenbasicsets}.
Since $x^{(n)}_{0} \to p \in W_I$ and $x^{(n)}_{\tau_n} \to q \in W_J$ as $n \to \infty$, there is $n_0 \geq 1$ such that
\[
	\max \{ \dist(x^{(n_0)}_{0}, p), \dist(x^{(n_0)}_{\tau}, q) \} < \Delta.
\]
Similarly, we may assume that $\Delta_{n_0} < \Delta$ and $\beta_n < \beta'$.
In particular,
\[
	(D_{n_0}, g_{n_0} \times \theta_{n_0}) \in \D(\Delta_{n_0}, \widetilde{W}^u(x^{(n_0)}_0, \beta) \times D^{k}) \subset \D(\Delta, \widetilde{W}^u(x^{(n_0)}_0, \beta) \times D^{k}).
\]
From these and Lemma \ref{lem.betweenbasicsets}, it follows that there exist $k' \geq 0$ and a continuous map $h' \times \eta': f^\tau(D) \to \widetilde{W}^u(x^{(n_0)}_{\tau}, \beta') \times D^{k'}$ with $\dim \widetilde{W}^u(x^{(n_0)}_{0}) + k = \dim \widetilde{W}^u(x^{(n_0)}_{\tau}) + k'$ such that
\[
	(f^\tau(D), h' \times \eta') \in \D(\alpha, \widetilde{W}^u(x^{(n_0)}_{\tau}, \beta') \times D^{k'})
\]
and
\[
	V^u(h' \times \eta') \subset f^\tau(V^u(g_{n_0} \times \theta_{n_0})).
\]
Applying Lemma \ref{lem.basicofepsilonclosever2} to $h' \times \eta'$, we have 
$h \times \eta: f^\tau(D) \to \widetilde{W}^u(x^{(n_0)}_{\tau}, \beta_{n_0}) \times D^{k'}$ with $\dim \widetilde{W}^u(x^{(n_0)}_{0}) + k = \dim \widetilde{W}^u(x^{(n_0)}_{\tau}) + k'$ such that 
\[
	(f^\tau(D), h \times \eta) \in \D(\alpha, \widetilde{W}^u(x^{(n_0)}_{\tau}, \beta_{n_0}) \times D^{k'})
\]
and
\[
	V^u(h \times \eta) \subset V^u(h' \times \eta').
\]
Thus, the proof of Proposition \ref{pro.betweenbasicsets} is complete.

Now, what is remaining is the proof of Proposition \ref{pro.eachstep}.
For the proof, fix any $\e > 0$ and $L \geq L_0$. 
Let $\tau > 0$ be such that $B(\tau, x) \subset U_i$ for all $x \in W_i$, $i=1, 2, \ldots, s$, where $U_i$ is a neighborhood of $\Lambda_i$ given in the beginning of Section \ref{sec.prepforthm.homologyshadowing} as $U_i=u$ and $\Lambda_i=\Lambda$.
Let $\alpha_0, \beta_0 > 0$ and $d_1 \in (0, d_0)$ be such that if $\xi = \{ x_k \}$ is a $d_1$-pseudotrajectory of $f$, then the following properties hold:
\begin{enumerate}[leftmargin=0.8cm, itemindent=0cm]
	\item[(a)] If there is $1 \leq i \leq s$ such that $x_k \in W_i$ for all $k \leq 0$, then there exists a point $y \in M$ such that
	      \[
		      \dist(f^k(y), x_k) < \min \{\e/2, \tau\}, \quad k \leq 0
	      \]
	      (this is possible since $f$ has the shadowing property in $W_i$ ($i=1, 2, \ldots, s$)).
	\item[(b)] If $x_j \in W_i$ and $w \in \B(x_j, \alpha_0, \beta_0)$ for some $i=1, 2, \ldots, s$ and $j \in \Z$, then
	      \begin{equation}
		      \dist(f^k(w), x_{j + k}) < \e, \quad 0 \leq k \leq L.\label{eq.finalstepshadowing}
	      \end{equation}
\end{enumerate}
We also assume that $\widetilde{W}^u(y, 2\beta_0) \subset B(y, \e/2)$ for all $y \in \bigcup_i U_i$.
From the choice of $\tau$, for $y$ given in (a) above, it follows that $f^k(y) \in U_i$ and $f^k(\widetilde{W}^u(y, 2\beta_0)) \subset \widetilde{W}^u(f^k(y), 2\beta_0)$ for all $k \leq 0$.
Thus, if $y' \in \widetilde{W}^u(y, 2\beta_0)$, then
\begin{equation}
	\dist(f^k(y'), x_k) \leq \dist(f^k(y'), f^k(y)) + \dist(f^k(y), x_k) < \e/2 + \e/2 = \e\label{eq.propertyofWy2beta_0}
\end{equation}
for all $k \leq 0$.
We apply Proposition \ref{pro.betweenbasicsets} to find numbers $\Delta = \Delta(\alpha_0, \beta_0, L)$, $d_2 = d(\alpha_0, \beta_0, L)$, $\beta' = \beta'(\alpha_0, \beta_0, L)$ and $N_0 = N(\alpha_0, \beta_0, L)$ (assuming $d_2 \leq d_1$).

After that, we find a number $\Delta_0 < \Delta(1 - \lambda)$ that satisfies inequality $\beta'/\lambda - \Delta_0 > \beta'/\mu$. Due to this choice of $\Delta_0$, we can find a number $N_1 \geq N_0$ such that
\begin{equation}
	\lambda^{N_1}\alpha_0 + \frac{\Delta_0}{1 - \lambda} < \Delta\label{eq.defofN1first}
\end{equation}
and
\begin{equation}
	\beta'\mu^{-N_1} > \beta_0,\label{eq.defofN1second}
\end{equation}
where $\mu \in (\lambda, 1)$ is given just before Proposition \ref{pro.generalizedcloseddiscchange}.
For each
\begin{equation}
	\beta \in \{ \beta', \beta'/\mu, \ldots, \beta'/\mu^{N_1}, \beta_0\},\label{eq.setofbeta}
\end{equation}
let $d(\beta) \in (0, d_2)$ be the number for which the conclusion of Proposition \ref{pro.generalizedcloseddiscchange} holds with $\alpha_0$, $\beta$, and $\Delta_0$.
Take $d_3 \in (0, d_2)$ as the minimum of $d(\beta)$ over all $\beta$ in \eqref{eq.setofbeta}.
We need the following claim.
\begin{claim}\label{claim.7}
	Making $d_3$ smaller if necessary, we may assume that if $d_3$-pseudotrajectory $\xi = \{ x_k \}$ satisfies
	\[
		x_k \in W_i, \quad k \leq 0,
	\]
	for some $1 \leq i \leq s$, then there exists $y \in U_i$ satisfying property (a) above holds and such that $\widetilde{W}^u(y, 2\beta_0) \in \D(\Delta, \widetilde{W}^u(x_0, \beta_0) \times D^0)$.
\end{claim}
Apply Lemma \ref{lem.purtaofbase} with $(\alpha, \beta, \beta', \Delta)$ replaced by $(\Delta/2, 2\beta_0, \beta_0, \Delta/2)$ and let $d > 0$ be the constant given in Lemma \ref{lem.purtaofbase}.
Since $\widetilde{W}^u(y, 2\beta_0) \in \D(\Delta/2, \widetilde{W}^u(y, 2\beta_0) \times D^0)$, if $y$ satisfies $\dist(x_0, y) < d$ then $\widetilde{W}^u(y, 2\beta_0) \in \D(\Delta, \widetilde{W}^u(x_0, \beta_0) \times D^0)$.
By using smaller $d_3$ if necessary, we can make $y$ that satisfies (a) as close to $x_0$ as we like.
This proves Claim \ref{claim.7}.

Notice that the fixed value $d_3$ depends on $\e$ and is independent of the choice of the pseudotrajectory $\{ x_i \}$.
Let $\xi \in PT(L, N_1, s)$ be a $d_3$-pseudotrajectory. Then let us prove that $\xi$ satisfies
\[
	\dist(x_k, f^k(x)) < \e, \quad k \in \Z.
\]
Since $\xi \in PT(L, N_1, s)$, there exist basic sets $\Lambda_1, \ldots, \Lambda_{\overline{s}}$ with $0 \leq \overline{s} \leq s$, $l \in \Z$ and
\begin{equation}
	0 = \tau(1) < t(2) \leq \tau(2) < t(3) \leq \tau(3) < \cdots < t(\overline{s}-1) \leq \tau(\overline{s}-1) < t(\overline{s})
\end{equation}
satisfying the four conditions in Definition \ref{dfn.PT(LKs)}.
To simplify the notation, without loss of generality, we may assume $l = 0$.

We start with the above-mentioned embedded closed ball $D = \widetilde{W}^u(y, 2\beta_0)$. There is a continuous map $h_0 \times \eta_0 : D \to \widetilde{W}^u(x_0, \beta_0) \times D^0$ with $(D, h_0 \times \eta_0) \in \D(\Delta, \widetilde{W}^u(x_0, \beta_0) \times D^0)$. Let $u_1 = \dim \widetilde{W}^u(x_0)$.
Since $t(2) - \tau(1) > N_1 \geq N_0$, by Proposition \ref{pro.betweenbasicsets}, we have $h_{t(2)} \times \eta_{t(2)} : f^{t(2)}(D) \to \widetilde{W}^u(x_{t(2)}, \beta') \times D^{u_1-u_2}$ satisfying $(f^{t(2)}(D), h_{t(2)} \times \eta_{t(2)}) \in \D(\alpha_0, \widetilde{W}^u(x_{t(2)}, \beta') \times D^{u_1-u_2})$ and
\[
	V^u(h_{t(2)} \times \eta_{t(2)}) \subset f^{t(2)}(V^u(h_0 \times \eta_0)),
\]
where $u_2 = \dim \widetilde{W}^u(x_{t(2)})$.

By the choice of $N_1$ (see the inequality \eqref{eq.defofN1second}),  there exists a number $\ell \in (0, N_1]$ such that
\begin{equation}
	\beta' \mu ^{-k} \leq \beta_0, \quad 0 \leq k < \ell, \quad \text{and} \quad \beta' \mu^{-\ell} > \beta_0.\label{eq.defofl}
\end{equation}
By Proposition \ref{pro.generalizedcloseddiscchange},
there is a continuous map $h_{t(2) + 1} \times \eta_{t(2) + 1} : f^{t(2) + 1}(D) \to \widetilde{W}^u(x_{t(2) + 1}, \beta'\mu^{-1}) \times D^{u_1-u_2}$ satisfying
\[
	(f^{t(2) + 1}(D), h_{t(2) + 1} \times \eta_{t(2) + 1}) \in \D(\lambda\alpha_0 + \Delta_0, \widetilde{W}^u(x_{t(2) + 1}, \beta'\mu^{-1}) \times D^{u_1-u_2})
\]
and
\[
	V^u(h_{t(2) + 1} \times \eta_{t(2) + 1})
	\subset f(V^u(h_{t(2)} \times \eta_{t(2)})).
\]

We continue this process and construct continuous maps
\[
	h_{t(2) + i} \times \eta_{t(2) + i} : f^{t(2) + i}(D) \to \widetilde{W}^u(x_{t(2) + i}, \beta'\mu^{-i}) \times D^{u_1-u_2}, \quad i = 0, 1, \ldots, \ell - 1,
\]
such that
\[
	(f^{t(2) + i}(D), h_{t(2) + i} \times \eta_{t(2) + i}) \in \D \left( \lambda^i\alpha_0 + \Delta_0\frac{1-\lambda}{1-\lambda^i}, \widetilde{W}^u(x_{t(2) + i}, \beta'\mu^{-i}) \times D^{u_1-u_2} \right)
\]
and
\[
	V^u(h_{t(2) + i} \times \eta_{t(2) + i}) \subset f(V^u(h_{t(2) + i - 1} \times \eta_{t(2) + i - 1}))
\]
for all $i = 1, 2, \ldots, \ell-1$.

By the choice of $\ell$ (see \eqref{eq.defofl}), using Proposition \ref{pro.generalizedcloseddiscchange} and Lemma \ref{lem.basicofepsilonclosever2}, we can find a continuous map $h_{t(2) + \ell} \times \eta_{t(2) + \ell} : f^{t(2) + \ell}(D) \to \widetilde{W}^u(x_{t(2) + \ell}, \beta_0) \times D^{u_1-u_2}$ such that
\[
	(f^{t(2) + \ell}(D), h_{t(2) + \ell} \times \eta_{t(2) + \ell}) \in \D(\lambda^\ell\alpha_0 + \Delta_0(1-\lambda)/(1-\lambda^\ell), \widetilde{W}^u(x_{t(2) + \ell}, \beta_0) \times D^{u_1-u_2})
\]
and
\[
	V^u(h_{t(2) + \ell} \times \eta_{t(2) + \ell}) \subset f(V^u(h_{t(2) + \ell-1} \times \eta_{t(2) + \ell-1})).
\]

Now continue the process described above by constructing continuous maps $h_{t(2) + i} \times \eta_{t(2) + i}$ such that
\[
	(f^{t(2) + i}(D), h_{t(2) + i} \times \eta_{t(2) + i}) \in \D \left( \lambda^{\ell+i}\alpha_0 + \Delta_0\frac{1-\lambda}{1-\lambda^{\ell+i}}, \widetilde{W}^u(x_{t(2) + \ell + i}, \beta_0) \times D^{u_1-u_2} \right)
\]
and
\[
	V^u(h_{t(2) + i} \times \eta_{t(2) + i})
	\subset f(V^u(h_{t(2) + i - 1} \times \eta_{t(2) + i - 1}))
\]
for all $i = \ell+1, \ell+2, \ldots, \tau(2) - t(2)$.
This process is stopped when we construct a continuous map
$h_{\tau(2)} \times \eta_{\tau(2)} : f^{\tau(2)}(D) \to \widetilde{W}^u(x_{\tau(2)}, \beta_0) \times D^{u_1-u_2}$ satisfying
\[
	(f^{\tau(2)}(D), h_{\tau(2)} \times \eta_{\tau(2)}) \in \D \left( \lambda^{\tau(2) - t(2)}\alpha_0 + \Delta_0\frac{1-\lambda}{1-\lambda^{\tau(2) - t(2)}}, \widetilde{W}^u(x_{\tau(2)}, \beta_0) \times D^{u_1-u_2} \right)
\]
and
\[
	V^u(h_{\tau(2)} \times \eta_{\tau(2)})
	\subset f (V^u(h_{\tau(2) - 1} \times \eta_{\tau(2) - 1})).
\]
Then, by \eqref{eq.defofN1first} and $\tau(2) - t(2) > N_1$, it follows that
\[
	(f^{\tau(2)}(D), h_{\tau(2)} \times \eta_{\tau(2)}) \in \D(\Delta, \widetilde{W}^u(x_{\tau(2)}, \beta_0) \times D^{u_1-u_2}).
\]
By our construction, we have
\begin{align}
	f^{-\tau(2)}(V^u(h_{\tau(2)} \times \eta_{\tau(2)}))
	 & \subset f^{-\tau(2) + 1}(V^u(h_{\tau(2) - 1} \times \eta_{\tau(2) - 1})) \\
	 & \subset \cdots                                                           \\
	 & \subset f^{-t(2)}(V^u(h_{t(2)} \times \eta_{t(2)}))                      \\
	 & \subset V^u(h_0 \times \eta_0).
\end{align}
Since $\tau(2) - t(2), \tau(3) - t(3) > N_1$,
we may apply Proposition \ref{pro.betweenbasicsets} to find a continuous map $h_{t(3)} \times \eta_{t(3)} : f^{t(3)}(D) \to \widetilde{W}^u(x_{t(3)}, \beta') \times D^{u_1-u_3}$ such that
\[
	(f^{t(3)}(D), h_{t(3)} \times \eta_{t(3)}) \in \D(\alpha_0, \widetilde{W}^u(x_{t(3)}, \beta') \times D^{u_1-u_3})
\]
and
\[
	V^u(h_{t(3)} \times \eta_{t(3)}) \subset f^{t(3) - \tau(2)}(V^u(h_{\tau(2)} \times \eta_{\tau(2)})),
\]
continuing this procedure.

As a consequence, we obtain a family of compact non-empty subsets of the initial closed ball $D$:
\[
	D_\ell = \overline{f^{-\ell}(V^u((h_{\ell} \times \eta_{\ell})))}
\]
indexed by the set
\begin{equation}
	\{ 0 \} \cup \left( \bigcup_{2 \leq i \leq \overline{s}-1}\{ t(i) \leq j \leq \tau(i) \} \right) \cup \{ j \geq t(\overline{s})\}.\label{eq.setofl}
\end{equation}
By the above construction, $D_{\ell_1} \subset D_{\ell_2}$ for all $l_1, l_2$ in \eqref{eq.setofl} with $l_1 > l_2$.
Let $z \in \bigcap_\ell D_\ell$, and let us claim
\begin{equation}
	\dist(x_\ell, f^\ell(z)) < \e\label{eq.conclusionofthm}
\end{equation}
for all $\ell \in \Z$.

Since $D = \widetilde{W}^u(y, 2\beta_0)$ with $y$ chosen so that property (a) holds,
\eqref{eq.propertyofWy2beta_0} implies \eqref{eq.conclusionofthm} for $\ell \leq 0$.

On the other hand, $f^\ell(z) \in f^\ell(D_\ell) = \overline{V^u((h_{\ell} \times \eta_{\ell}))}$ for all $\ell$ with \eqref{eq.setofl}.
Thus,
\[
	\widetilde{W}^s(f^\ell(z), \alpha_0) \cap \widetilde{W}^u(x_\ell, \beta_0) \neq \emptyset
\]
(see Remark \ref{rk.motivationofD(a(w(xb)))}),
and therefore $f^\ell(z) \in \B(x_\ell, \alpha_0, \beta_0)$.
Now apply \eqref{eq.finalstepshadowing} for $k = 0$ to obtain \eqref{eq.conclusionofthm}.

Finally, for $\ell=\ell_0 + k_0$, with $\ell_0$ in \eqref{eq.setofl} and $0 \leq k_0 \leq L$,
we can prove $f^{\ell_0}(z) \in \B(x_{\ell_0}, \alpha_0, \beta_0)$ as above,
and then apply \eqref{eq.finalstepshadowing} again for $k = k_0$ to obtain \eqref{eq.conclusionofthm}.

\section{The proof of Theorem \ref{thm.3dim}}
It is obvious that Theorem \ref{thm.3dim} is obtained from Theorem \ref{thm.homologyshadowing} and the following proposition.
Therefore, we prove Proposition \ref{pro.lowdim} in this section.
\begin{pro}\label{pro.lowdim}
	Let $f$ be an Axiom A diffeomorphism.
	Suppose that
	\[
		\dim W^s(x), \dim W^u(x) = 0, 1, m-1, m
	\]
	for all $ x \in M$ and $f$ satisfies the $C^0$ transversality condition.
	Then $f$ satisfies both the $T^{s, u}$ and $T^{u, s}$-condition.
\end{pro}
\begin{proof}
	Let us prove that $f$ satisfies the $T^{s, u}$-condition (the proof of the case where $f$ satisfies the $T^{u, s}$-condition is similar).
	If $x \in M$ satisfies $\dim W^s(x)=m$ or $\dim W^u(x) = m$,
	then $x$ belongs to a transversal intersection of $W^s(x)$ and $W^u(x)$.
	Thus, Remark \ref{rk.homology} (2) implies that $W^s(x)$ and $W^u(x)$ satisfy the $T$-condition at $x$.
	Note that if $x$ satisfies $\dim W^s(x)=0$ (resp. $\dim W^u(x) = 0$),
	then $x \in \Omega(f)$ and $\dim W^u(x)=m$ (resp. $\dim W^s(x) = m$).
	Using Remark \ref{rk.homology} (1),
	we have
	\[\dim W^s(x) + \dim W^u(x) \geq m\]
	for all $x \in M$.
	Thus, it remains to be proved is that for every $x \in M$ satisfying
	\[
		(\dim W^s(x), \dim W^u(x)) = (1, m-1), (m-1, 1), (m-1, m-1),
	\]
	$W^s(x)$ and $W^u(x)$ satisfy the $T$-condition at $x$.
	Let us consider the following two cases separately.
	\vspace{2mm}\\
	{\bf Case 1:} $(\dim W^s(x), \dim W^u(x)) = (m-1, 1), (m-1, m-1)$.
	
	Let $h_s : \R^{m-1} \to W^s(x)$ and $h_u : \R^u \to W^u(x)$ be the immersions of $W^s(x)$ and $W^u(x)$,
	where $u = 1, m-1$.
	Assume $h_s(0) = h_u(0)  = x$.
	Let $U_s \subset \R^{m-1}$ and $U_u \subset \R^u$ be arbitrary neighborhoods of the origin.
	Let $U \subset M$ be a neighborhood of $x$ and let $\phi : U \to (-1, 1)^m$ be a local chart of $(-1, 1)^{m-1}$ around $\{0\}^{m-1}$
	(i.e., $\phi(h_s(U_s) \cap U) = (-1, 1)^{m-1} \times \{ 0 \}$).
	Making $U_u$ smaller if necessary,
	we may assume that $h_u(U_u) \subset U$ and take $i : U_u \to \R$ as
	\[
		i = \pi_m \circ \phi \circ h_u,
	\]
	where $\pi_m : \R^m \to \R$ is the projection to the $m^{\text{th}}$ component.
	We need the following claim:
	\begin{claim}\label{claim.lowdim}
		If $v_0 \in U_u$ satisfies $i(v_0) = 0$ and $U_u' \subset U_u$ is  a neighborhood of $v_0$, then there exist $v_1, v_2 \in U_u'$ such that $i(v_1) > 0$ and $i(v_2) < 0$.
	\end{claim}
	Let $v_0 \in U_u$ with $i(v_0) = 0$.
	Assume to the contrary that there exists a neighborhood $U_u' \subset U_u$ of $v_0$ with $i(v) \geq 0$
	for all $v \in U_u'$ (the proof for the case where $i(v) \leq 0$
	for all $v \in U_u'$ is similar).
	Making $U_u'$ smaller if necessary, we may also assume that 
	\begin{equation}
		i(v) \in [0, 1/2]\label{eq.condoni}
	\end{equation}
	for all $v \in U_u'$.
	Let $V_u \subset U_u'$ be a neighborhood of $v_0$ with $\overline{V_u} \subset U_u'$ and
	let $\tau : \R^u \to [0, 1]$ be a $C^\infty$ function such that ${\rm Supp}(\tau) \subset U_u'$ and $\tau(v) = 1$ for all $v \in V_u$.
	Then there exists $C_1 > 0$ such that, for all $\e \in  (0, 1/2)$,
	a continuous function $\widetilde{h}_u : \R^u \to M$ defined by
	\[
		\widetilde{h}_u(y) = 
		\begin{cases}
			\phi^{-1}(\phi \circ h_u(y) + (0, 0, \ldots, 0, \e\tau(y))), & \quad y \in U_u',           \\
			h_u(y),                                                      & y \in \R^u \setminus U_u',
		\end{cases}
	\]
	satisfies $\lvert h_u, \widetilde{h_u} \rvert_{C^0} < C_1\e$.
	Then
	\[
		\pi_m \circ \phi \circ \widetilde{h_u} (v)
		= \pi_m \circ \phi \circ h_u (v)
		+ \e\tau(v)
		= i(v) + \e\tau(v) > 0
	\]
	for all $v \in V_u$.
	On the other hand, $\pi_m \circ \phi \circ h_s (U_s) = 0$.
	Thus, we have $h_s(U_s) \cap \widetilde{h_u}(V_u) = \emptyset$.
	This contradicts the $C^0$-transversality condition of $f$, which proves the claim.
	
	Since $i(\{0\}^u) = 0$, Claim \ref{claim.lowdim} implies that
	there are $v_1, v_2 \in U_u$ such that $i(v_1) > 0$ and $i(v_2) < 0$.
	Here we can suppose
	\begin{equation}
		v_1 \notin \{tv_2 \in \R^u ; t \geq 0\}.\label{eq.xyarenotlinear}
	\end{equation}
	For otherwise,
	take $v \in U_u \setminus \{tv_2 \in \R^u ; t \geq 0\}$.
	If $i(v) > 0$, then we have \eqref{eq.xyarenotlinear} with $v_1$ replaced by $v$.
	Similarly, if $i(v) < 0$, then replace $v_2$ by $v$, getting \eqref{eq.xyarenotlinear}.
	If $i(v) = 0$, then apply Claim \ref{claim.lowdim} with $v_0 = v$ to have $v' \in U_u \setminus \{tv_2 \in \R^u ; t \geq 0\}$ such that $i(v') > 0$.
	By replacing $v_1$ by $v'$, we obtain \eqref{eq.xyarenotlinear}.
	Then 
	we may take an embedded closed interval $B_u \subset U_u$ centered at the origin whose two endpoints $v_1$, $v_2$
	satisfy $i(v_1) > 0$ and $i(v_2) < 0$.
	
	This implies $\phi \circ h_u (\partial B_u) \subset (-1, 1)^m \setminus (-1, 1)^{m-1} \times \{0\}$
	and hence
	\[
		(\phi \circ h_u)_* : \widetilde{H}_{0}(\partial B_u) \to \widetilde{H}_{0}((-1, 1)^m \setminus (-1, 1)^{m-1} \times \{0\})
	\]
	is nontrivial.
	\vspace{2mm}\\
	{\bf Case 2:} $(\dim W^s(x), \dim W^u(x)) = (1, m-1)$.
	
	Let $h_s : \R \to W^s(x)$ and $h_u : \R^{m-1} \to W^u(x)$ be the immersions of $W^s(x)$ and $W^u(x)$.
	Assume $h_s(0) = h_u(0) = x$.
	Let $U_s \subset \R$ and $U_u \subset \R^{m-1}$ be arbitrary neighborhoods of the origin.
	Let $U$ be a neighborhood of $x$ and let $\phi : U \to (-1, 1)^m$ be a local chart of $(-1, 1)$ around \{0\} (i.e., $\phi(h_s(U_s) \cap U) = (-1, 1) \times \{ 0 \}^{m-1}$).
	If $W^s(x)$ and $W^u(x)$ are transversal at $x$, then by Remark \ref{rk.homology} (2),
	there is nothing to prove.
	Thus, without loss of generality, we may assume that $W^s(x)$ and $W^u(x)$ do not have a transversal intersection at $x$ and
	$\phi \circ h_u(U^u)$ is tangent to $(-1, 1)^{m-1} \times \{0\}$ at $\{0\}^m$.
	Let $\pr_m : \R^m \to \R^m$ be the projection defined by
	\[
		\pr_m(x_1, x_2, \ldots, x_m) = (x_1, x_2, \ldots, x_{m-1}, 0).
	\]
	Making $U_u$ smaller if necessary, we may assume that $\pr_m \vert_{\phi (h_u(U_u) \cap U)} : \phi (h_u(U_u) \cap U) \to (-1, 1)^{m-1} \times \{0\}$ is a diffeomorphism onto its image.
	Let $g : \pr_m \circ \phi(h_u(U_u) \cap U) \to (-1, 1)$ be the $m$-th component of the inverse of $\pr_m \vert_{\phi (h_u(U_u) \cap U)}$,
	which satisfies
	\[
		(x_1, x_2, x_3, \ldots, g(x)) = (\pr_m \vert_{\phi (h_u(U_u) \cap U)})^{-1}(x)
	\]
	for all $x = (x_1, x_2, \ldots, x_{m-1}, 0) \in \pr_m \circ \phi(h_u(U_u) \cap U)$.
	We can take $s_1 < 0 < s_2$ so that $g((s_1, 0, 0, \ldots, 0)) > 0$ and
	$g((s_2, 0, 0, \ldots, 0)) < 0$.
	In fact, suppose by contradiction that there exist $t_1 < 0 < t_2$ such that $g((t, 0, 0, \ldots, 0)) \geq 0$ for all $t \in (t_1, t_2)$.
	Then it is easy to exhibit a contradiction from the argument as in the proof of Claim \ref{claim.lowdim}.
	
	Take sufficiently small $\e > 0$ so that
	$g(x) > 0$ if $x$ is in the $\e$-neighborhood of $(s_1, 0, 0, \ldots, 0)$ and
	$g(y) < 0$ if $y$ is in the $\e$-neighborhood of $(s_2, 0, 0, \ldots, 0)$.
	Let
	\begin{align}
		S_1 & = \{ (s_1,x_2, x_3, x_4 \ldots, x_{m-1}, 0) ; (x_2, x_3 \ldots, x_{m-1}) \in  B_{\R^{m-2}}(\e) \},                                 \\
		S_2 & = \{ (s_2,x_2, x_3, x_4 \ldots, x_{m-1}, 0) ; (x_2, x_3 \ldots, x_{m-1}) \in  B_{\R^{m-2}}(\e) \}, \quad \text{ and }              \\
		C_3 & = \{ (x_1,x_2, x_3, x_4 \ldots, x_{m-1}, 0) ; s_1 \leq x_1 \leq s_2, (x_2, x_3 \ldots, x_{m-1}) \in \partial B_{\R^{m-2}}(\e) \},
	\end{align}
	where $B_{\R^{m-2}}(\e) \subset \R^{m-2}$ is a closed $(m-2)$-dimensional ball with radius $\e$ (see Figure \ref{figure.zu_case2}).
	\begin{figure}[h]
		\includegraphics[width=15cm]{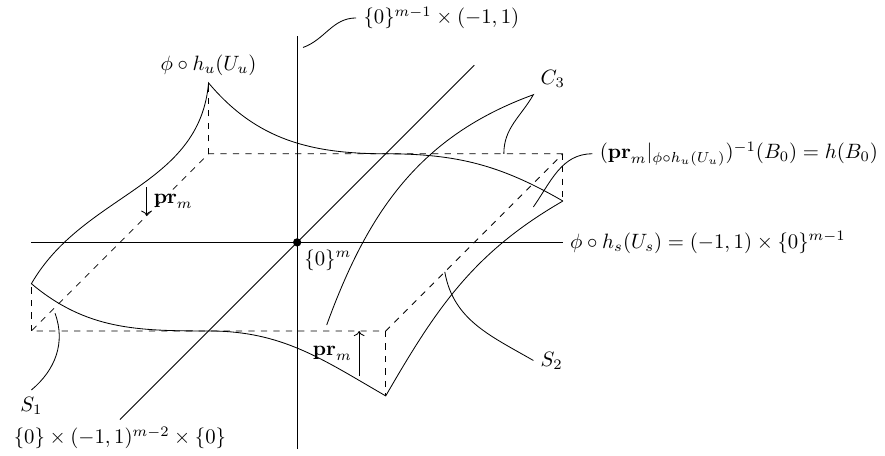}
		\caption{The map $h$}\label{figure.zu_case2}
	\end{figure}
	Define $S_0 \subset (-1, 1)^m$ by
	\[
		S_0 = S_1 \cup S_2 \cup C_3 \approx S^{m-2}.
	\]
	Then $S_0$ is the boundary of
	\[
		B_0 = \{ (x_1, x_2, x_3, x_4 \ldots, x_{m-1}, 0) ; s_1 \leq x_1 \leq s_2, (x_2, x_3 \ldots, x_{m-1}) \in B_{\R^{m-2}}(\e) \}.
	\]
	Consider
	a diffeomorphism $h : B_0 \to (-1, 1)^m$ onto its image such that
	\begin{align}
		h(x_1, x_2, x_3, \ldots, x_{m-1}, 0)
		&= (x_1, x_2, x_3, \ldots, x_{m-1}, g(x))\\
		&= (\pr_m \vert_{\phi (h_u(U_u) \cap U)})^{-1}(x).
	\end{align}
	
	\begin{claim}\label{claim.9}
		We have $h(S_0) \subset (-1, 1)^m \setminus (-1, 1) \times \{0\}^{m-1}$ and 
		$h \vert_{S_0}$ induces a nontrivial homomorphism between homology groups
		\[
			(h \vert_{S_0})_* : \widetilde{H}_{m-2}(S_0) \to \widetilde{H}_{m-2}((-1, 1)^m \setminus (-1, 1) \times \{0\}^{m-1}).
		\]
	\end{claim}
	Define $\tilde{g} : S_0 \to (-1, 1)$ by
	\[
		\tilde{g}(x) = \frac{g(s_2, 0, 0, \ldots, 0) - g(s_1, 0, 0, \ldots, 0)}{s_2 - s_1}(x_1-s_1) + g(s_1, 0, 0, \ldots, 0).
	\]
	Then
	\[
		\tilde{g}(x) = 
		\begin{cases}
			g(s_1, 0, \ldots, 0), & x \in S_1,  \\
			g(s_2, 0, \ldots, 0), & x \in S_2.
		\end{cases}
	\]
	Let $G : [0, 1] \times S_0 \to (-1, 1)$ be a map defined by
	\[
		G(t, x) = (x_1, \ldots, x_{m-1}, (1-t)g(x) + t\tilde{g}(x)).
	\]
	Then $G(t, x) > 0$ if $x \in S_1$ and $G(t, x) < 0$ if $x \in S_2$.
	Thus, 
	$G$ is a homotopy map between $h \vert_{S_0}$ and $\tilde{h} : S_0 \to (-1, 1)^m \setminus (-1, 1) \times \{0\}^{m-1}$ defined by 
	\[
		\tilde{h}(x_1, x_2, x_3, \ldots, x_{m-1}, 0) = (x_1, x_2, x_3, \ldots, x_{m-1}, \tilde{g}(x)).
	\]
	It is easy to see that $\tilde{h}$ induces a nontrivial homomorphism between the homology groups, which proves Claim \ref{claim.9}.
	
	Thus, the map $\phi^{-1} \circ h \vert_{S_0}$ satisfies the definition of $T^{s,u}$-condition, which completes the proof of Proposition \ref{pro.lowdim}.
\end{proof}
\section*{acknowledgements}
The author is grateful to my advisor S. Hayashi for his constructive suggestions and continuous support.
This work was supported by the Research Institute for Mathematical Sciences,
an International Joint Usage/Research Center located in Kyoto University.

\bibliography{math}
\bibliographystyle{amsplain}
\end{document}